\documentclass{article}

\usepackage[utf8]{inputenc}
\usepackage{microtype}
\usepackage{graphicx}
\usepackage{subcaption}
\usepackage{caption}
\usepackage{booktabs} 
\usepackage[margin=1.0in]{geometry}
\usepackage{mathtools}
\usepackage{amsfonts}           
\usepackage{amsthm}
\usepackage{amsmath}
\usepackage{amssymb}
\usepackage{stmaryrd}
\usepackage{xcolor}
\usepackage{caption}
\usepackage{algorithm}
\usepackage{algpseudocode}
\usepackage{comment}
\usepackage{authblk}
\usepackage{hyperref}
\usepackage{multirow}

\newcommand{\ones}{\mathbf{1}}

\newcommand{\disp}{\mathrm{dis}_p^{\mathrm{GW}}(\pi)}
\newcommand{\gw}{d_{\mathrm{GW},p}}
\newcommand{\discp}{\mathrm{dis}^{\mathrm{COOT}}_p}
\newcommand{\Proj}{\operatorname{Proj}}
\usepackage{faktor}
\newcommand{\coot}{d_{\mathrm{COOT},p}}

\newcommand{\R}{\mathbb{R}}
\newcommand{\TpOT}{\mathrm{TpOT}}
\newcommand{\PD}{\mathrm{PD}}
\newcommand{\MPD}{\mathrm{MPD}}
\newcommand{\GW}{\mathrm{GW}}
\newcommand{\adm}{\mathrm{adm}}
\newcommand{\KL}{\operatorname{KL}}
\newcommand{\Lc}{\mathsf{L}}
\renewcommand{\tilde}{\widetilde}
\newcommand{\diff}{\mathrm{d}}

\newcommand{\rrrev}[1]{{#1}}

\theoremstyle{thmstyleone}
\newtheorem{theorem}{Theorem}
\newtheorem{proposition}[theorem]{Proposition}

\theoremstyle{thmstyletwo}
\newtheorem{example}{Example}
\newtheorem{remark}{Remark}
\newtheorem{lemma}{Lemma}
\newtheorem{corollary}{Corollary}

\theoremstyle{thmstylethree}
\newtheorem{definition}{Definition}

\begin{document}

\title{Topological Optimal Transport for Geometric Cycle Matching}

\author[1,2]{Stephen Y Zhang}
\author[1,2,3]{Michael P H Stumpf}
\author[4]{Tom Needham}
\author[5]{Agnese Barbensi}
\affil[1]{\small School of Mathematics and Statistics, the University of Melbourne, Parkville, Melbourne, 3010, VIC, Australia}
\affil[2]{\small Melbourne Integrative Genomics, the University of Melbourne, Parkville, Melbourne, 3010, VIC, Australia}
\affil[3]{\small School of Biosciences, the University of Melbourne, Parkville, Melbourne, 3010, VIC, Australia}
\affil[4]{\small Mathematics Department, Florida State University, Academic Way, Tallahassee, 32306, FL, USA}
\affil[5]{\small School of Maths and Physics, University of Queensland, St Lucia, Brisbane, 4067, QLD, Australia}

\maketitle

\begin{abstract}
  Topological data analysis is a powerful tool for describing topological signatures in real world data. An important challenge in topological data analysis is the task of matching significant topological signals across distinct systems. Optimal transportation provides a geometric and probabilistic approach to formalising notions of distances and matchings between measures and structured objects more generally.
  Building upon recent advances in the domains of persistent homology and optimal transport for hypergraphs, we develop an approach for finding geometrically and topologically informed matchings between point cloud datasets and their topological features.
  We define measure topological networks, which integrate both geometric and topological information about a system, and introduce a distance on the space of these objects.
  Our distance is defined in terms of a Topological Optimal Transport (TpOT) problem which seeks to transport mass that preserves geometric as well as topological relations, in the sense of minimising the induced geometric and topological distortions.
  We study the metric properties of this distance and show that it induces a geodesic metric space of non-negative curvature. We demonstrate our approach in a number of numerical experiments.
\end{abstract}

\section{Introduction}\label{sec1}

Topological data analysis (TDA) is a quickly growing field in computational and applied topology. In recent years, TDA has established itself as an effective framework to analyse, cluster, and detect patterns in complex data~\cite{wasserman2018topological}. One of the key algorithms in TDA is persistent homology (PH), a computational tool that describes the structure of data based on topological features persisting across different scales~\cite{otter2017roadmap, edelsbrunner2010computational, ghrist2008barcodes}. In brief, PH computation proceeds by building a nested sequence of discrete spaces describing the input data at increasingly coarse scales, known as a \textit{filtration} of simplicial complexes. The topological features of this filtration are then quantified by computing the homology groups of each simplicial complex in the sequence. A \emph{structural theorem}~\cite{zomorodian2012topological} guarantees that the birth, death, and evolution of homology classes in the filtration can be summarised as a multi-set, called a \textit{persistence diagram} (PD). PDs have been shown to contain rich information about the initial data, and PH-based data analysis approaches appear in an ever-increasing number of applications across multiple fields in modern science~\cite{saggar2018towards, sorensen2020revealing, vipond2021multiparameter, thorne2022topological}. 

Optimal transport (OT) is a mathematical framework that, in its classical formulation, formalises the problem of finding a matching between two given probability distributions that is optimal in terms of a prescribed cost~\cite{villani2009optimal, peyre2019computational}. When the cost function is induced by the distance between points in a metric space $(X,d)$, optimal transport provides a natural lifting of the ground metric to a metric on the space of probability distributions supported on $X$, commonly known as the \textit{Wasserstein} distance (see for instance, \cite[Theorem 6.18]{villani2009optimal}). An extension of the optimal transportation problem which has recently received considerable attention is the \emph{Gromov-Wasserstein} (GW) problem~\cite{memoli2007,memoli2011gromov, chowdhury2019gromov,sturm2023space}, in which the probability distributions to be matched are defined over different metric spaces, rather than supported on a common space. The objective is then to find matchings that minimise \emph{distortion} of pairwise distances. While the GW distance was originally introduced to compare metric measure spaces and can be viewed as a relaxation of the well known Gromov-Hausdorff distance~\cite{memoli2007}, it has seen recent success in matching more general structured objects such as (labelled) graphs~\cite{xu2019gromov, xu2019scalable, chowdhury2020gromov, vayer2020fused}. Among other advantages, GW distances and matchings have brought substantial improvements to graph processing tasks over classical methods, and can be efficiently approximated by a range of numerical algorithms in practice~\cite{xu2019scalable,peyre2016gromov,flamary2021pot,chowdhury2021quantized}. On the theoretical front, the space of \emph{gauged measure spaces} (i.e., measure spaces $(X, \mu)$ together with a symmetric gauge function $k : X \times X \to \mathbb{R}$) endowed with the GW distance has the geometry of an Alexandrov space, and a Riemannian orbifold structure can be developed in this setting~\cite{sturm2023space, chowdhury2020gromov}. More recently, inspired by the goal of encoding relations in complex systems as higher-order networks, a variant of the GW problem, referred to as \emph{co-optimal transport}~\cite{redko2020co}, was applied to model hypergraphs both from a theoretical and applied point of view~\cite{chowdhury2023hypergraph}. 

\begin{figure}[ht]
    \centering
    \includegraphics[width=0.6\textwidth]{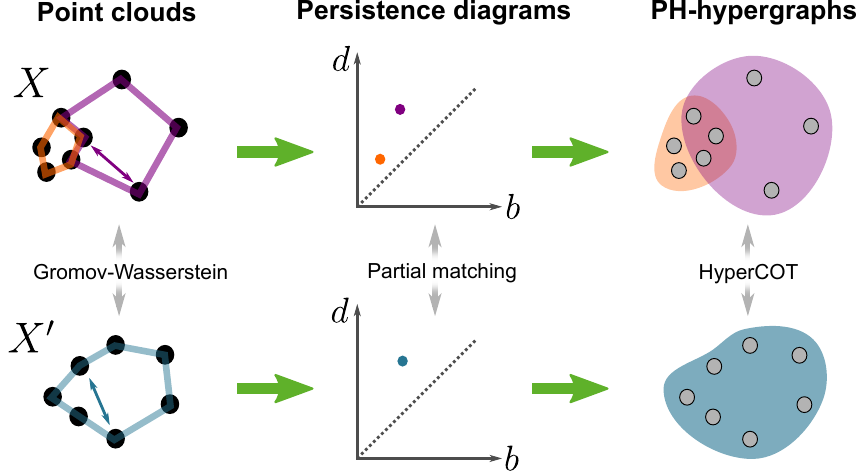}
    \caption{\textbf{Topological Optimal Transport (TpOT).} From left to right: (i) two input point clouds, (ii) their persistence diagrams, and (iii) corresponding PH-hypergraphs. The objective of the TpOT problem shares information across these three levels via a combination of distortion functionals: a Gromov-Wasserstein distortion on the point clouds (geometric information), (partial) Wasserstein matching of points in the persistence diagrams (topological information) and HyperCOT on the PH-hypergraphs (coupling of geometric and topological information). }\label{fig:tpot_pipeline}
\end{figure}

In this paper, we combine the TDA and OT approaches described above to develop a transport-based method for topological data processing. Our approach is inspired by recent extensions built upon persistent homology computations: the hyperTDA framework~\cite{barbensi2022hypergraphs} encodes each persistence feature -- represented as a point in a persistence diagram -- together with a \emph{geometric} realisation of its corresponding homology class as a generating cycle in the input point cloud. The resulting object is a higher-order network called a \emph{PH-hypergraph}. In a PH-hypergraph, vertices correspond to data points used as input to the PH computation, and hyperedges are given by the generators of features in the persistence diagram. We define the \emph{topological optimal transport (TpOT) problem}, by considering a trade-off between preservation of geometric relationships (encoded in pairwise proximity of points in each point cloud), preservation of topological features (encoded in the PH-hypergraph), and a coupling of the two via the PH-hypergraph structure (see Figure~\ref{fig:tpot_pipeline}). The output is \textbf{(1)} a matching between points that is geometrically driven and topologically informed, together with \textbf{(2)} a matching between persistent homology classes that is topologically driven and geometrically informed. 

Finding meaningful ways to match topological features across distinct systems is an important challenge in TDA. In recent years there has been considerable effort towards addressing this question, with a number of different solutions proposed~\cite{cohen2009persistent, bauer2013induced, gonzalez2020basis, reani2022cycle,yoon2023persistent, garcia2022fast}, many of which bridge the topological data analysis and optimal transport literatures~\cite{giusti2023signatures, hiraoka2023topological, li2023flexible}. A key aspect of the TpOT method is leveraging \emph{existing} approaches for optimal matchings of persistence diagrams to drive a matching between topological features, while also taking into account geometric information and the spatial connectivity of homology generators. By endowing PH-hypergraphs with probability measures, we introduce \textit{measure topological networks} as a structure for representing both geometric and PH information, and we formalise a framework for studying the TpOT induced distance $d_{\TpOT, p}$. 

Our contributions are as follows:
\begin{itemize}
    \item We develop a flexible measure-theoretic formalism for simultaneously encoding the geometry and topology of a point cloud, in the form of \emph{measure topological networks}, as well as a family of distances $d_{\TpOT,p}, p \geq 1$, between these objects; see Definitions~\ref{def:measure_topological_network} and \ref{def:TpOT}, respectively.
    \item The distance $d_{\TpOT, p}$ is shown to be a pseudometric on the space of measure topological networks $\mathcal{P}$, and the zero distance equivalence relation is completely characterised; see Theorem \ref{thm:metric}. 
    \item We show that the metric induced by $d_{\TpOT, p}$ is geodesic, characterise the exact form of the geodesics, and use this characterisation to show that this metric space is non-negatively curved in the sense of Alexandrov; see Theorems \ref{thm:geosame}, \ref{thm:geounique} and \ref{thm:nonnegative_curvature}, respectively.
    \item Our framework is formulated in terms of a non-convex optimisation problem, and we provide numerical algorithms for its approximate solution. We demonstrate its application in practice on a variety of examples; see Section \ref{sec: results}.
\end{itemize}

The paper is organised as follows. Section~\ref{sec:maths} describes the necessary mathematical background on persistent homology and optimal transport. In Section~\ref{sec:tpot}, after motivating and constructing the TpOT problem, we define \emph{measure topological networks} and the family of (pseudo-)metrics $d_{\TpOT, p}$. We then discuss theoretical properties (Subsection~\ref{sec:metric}) of the distance $d_{\TpOT, p}$, and provide a characterisation of geodesics in Section~\ref{sec:characterization_of_geodesics}. Finally, we provide details of the computational implementation (Subsection~\ref{sec:solving}), as well as examples (Section~\ref{sec: results}). 

\section{Mathematical background}~\label{sec:maths}

Our formalisation of the TpOT problem relies on constructions in topological data analysis and a measure theoretic formalism. The aim of this section is to introduce the mathematical background needed to define measure topological networks and the family of distances $d_{\TpOT, p}$ described in the Introduction.

\subsection{Persistent homology and persistent homology-hypergraphs}
Let $(X,d)$ be a finite metric space -- for example, a point cloud $X = \{ x_1, \cdots, x_N \} \subset \mathbb{R}^n$ as illustrated in Figure~\ref{fig:ph}(a). For any $\varepsilon > 0$, the \textit{Vietoris-Rips complex} $K_\varepsilon(X)$ (see \textit{e.g.}~\cite[Ch.III.2]{edelsbrunner2010computational}) is the simplicial complex obtained from $X$ by adding a $k$-simplex $[x_{i_0}, x_{i_1}, \cdots ,x_{i_{k}}]$ whenever the distance between all pairs of points in $\{x_{i_0}, x_{i_1}, \cdots ,x_{i_{k}}\}$ is less than $\varepsilon$. Note that $K_{\varepsilon}(X)$ is a sub-complex of $K_{\varepsilon'}(X)$ whenever $\varepsilon \leq \varepsilon'$. Thus, as $\varepsilon$ grows, this yields a nested sequence of simplicial complexes:
\[
  \mathcal{K}_{\rm VR}(X) := K_{\varepsilon_0}(X) \hookrightarrow K_{\varepsilon_1}(X) \hookrightarrow \cdots \hookrightarrow K_{\varepsilon_M}(X).
\]
Here, the numbers $\varepsilon_i$, $i=1,\ldots,M$, are the parameters corresponding to the creation of new simplices; see Figure~\ref{fig:ph}(b). 
\begin{figure}[ht]
  \centering
  \includegraphics[width=0.7\textwidth]{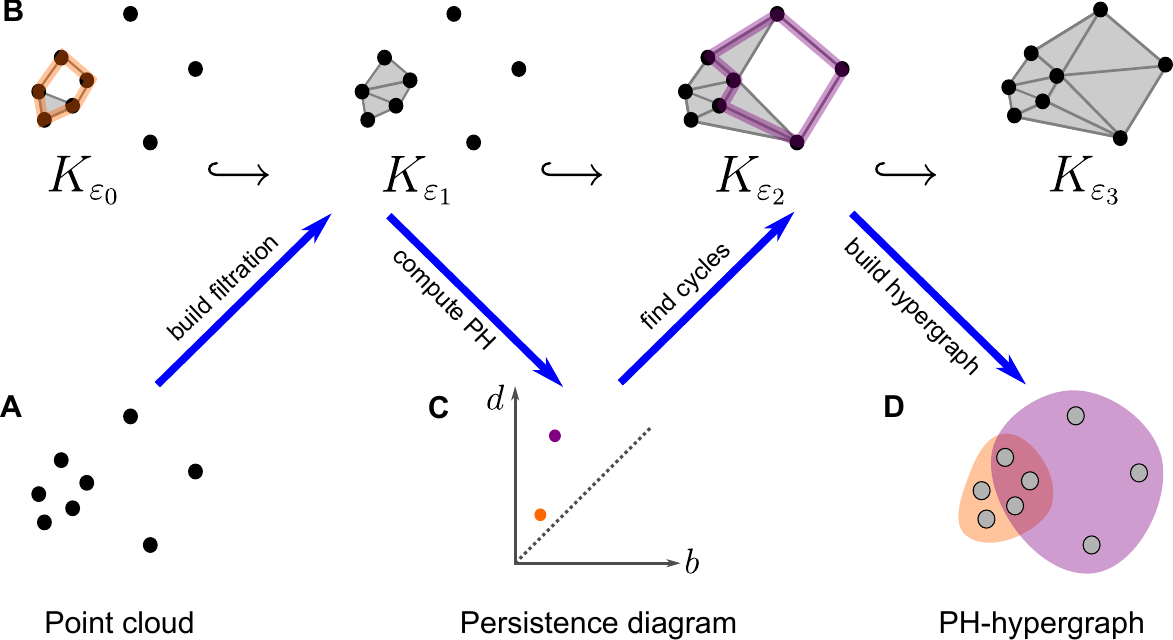}
  \caption{\textbf{Computing PH and the PH-hypergraph} \textbf{(a)} A point cloud and \textbf{(b)} a filtration of simplicial complexes built on it. \textbf{(c)} The $1$-dimensional persistence diagram of the filtration in (b). \textbf{(d)} The PH-hypergraph constructed after computing a representative cycle (orange and purple cycles in (b)) for each class in (c).}\label{fig:ph}
\end{figure}

More generally, for a finite metric space $(X, d)$, a \emph{filtration} over $X$ is a finite sequence $\mathcal{K} = (K_{\varepsilon_1}, \ldots, K_{\varepsilon_M})$ of simplicial complexes, indexed by an increasing sequence of real numbers $\varepsilon_i \leq \varepsilon_{i+1}$, together with simplicial maps $K_{\varepsilon_i} \to K_{\varepsilon_{i+1}}$ for all $i$, such that the vertex set of each $K_{\varepsilon_i}$ is a subset of $X$. Common alternatives to the Vietoris-Rips complex in the literature with a similar interpretation include the \v{C}ech complex~\cite[Ch.III.2]{edelsbrunner2010computational} and the Alpha complex~\cite[Ch.III.4]{edelsbrunner2010computational}. Different constructions of the filtration $\mathcal{K}$ will influence the results of the persistent homology computation and their interpretation  -- our framework is general and does not depend on a specific filtration construction, as long as generator representatives can be retrieved computationally.

A filtration $\mathcal{K}$ over $X$ can be analysed by computing the simplicial homology groups of each simplicial complex in the sequence and the corresponding induced maps. Collectively, this information is called the \emph{persistent homology} (PH) of $\mathcal{K}$. The Structural Theorem~\cite{zomorodian2005computing} guarantees that, when using coefficients in $\mathbb{Z}/2\mathbb{Z}$, the birth, death, and evolution of homology classes in a filtration of simplicial complexes can be summarised in a multi-set called a \emph{persistence diagram}. Let $\Lambda$ denote the upper-diagonal cone in the positive quadrant, $\Lambda = \{ (x, y) \in \mathbb{R}^2, y > x > 0 \}$. The persistence diagram $D$ associated to the filtration $\mathcal{K}$ is then a collection of points in $\Lambda$:
\[
  D = \left\{ (b_i,d_i) \right\}_{i = 1}^{|D|} \: \subset \: \Lambda. 
\]

Each point $(b_i, d_i) \in D$ represents a homology class in the filtration $\mathcal{K}$, with coordinates corresponding to the birth $b_i$ and death $d_i$ parameters of the class (see Figure~\ref{fig:ph}(c)). The \textit{persistence} $d_i-b_i > 0$ of a homology class is a measure of significance of the underlying topological feature. The persistence diagram $D$ can be interpreted as a topological descriptor of $X$. A local, geometric realisation of each homology class can be obtained by computing a representative cycle generating the class: see for instance the orange and purple cycles in Figure~\ref{fig:ph}(b). Computing cycles depends on a number of choices~\cite{ChenGenerators, Dey2019,  Obayashi2018VolumeOC, Ripser_involuted, LiMinimalCycle}, and therefore different representatives may lead to differing downstream results.

The recently introduced PH-hypergraph~\cite{barbensi2022hypergraphs} encodes the algebraic information of a persistence diagram, together with the local geometric interpretation of generating cycles, in a single higher-order network. Recall that a \emph{hypergraph} $(V, E)$ consists of a vertex set $V$ and a set of hyperedges $E$, such that each hyperedge $e \in E$ is a subset of $V$, i.e. $e \subseteq V$. Given a point cloud $X$, we may construct an associated filtration $\mathcal{K}$ together with a choice $g = \{c_1, \ldots, c_{|D|}\}$ of generating cycles for classes in the corresponding persistence diagram $D$. These data can be summarised in a PH-hypergraph $H = H(X, \mathcal{K}, g)$ (see Figure~\ref{fig:ph}(d)) in which we take $X$ to be the vertex set, and one hyperedge added for each cycle in $g$. Specifically, for every cycle $c \in g$, a point $x_i \in X$ is contained in the corresponding hyperedge $e(c)$ if it is a vertex of any simplex in $c$. While the hypergraph structure $H$ clearly depends on the choice of generators $g$ and the choice of filtration $\mathcal{K}$, for a fixed filtration $\mathcal{K}$ the qualitative (hyper)network structure of the PH-hypergraph has been empirically shown to be stable to noise and different choices of generating cycles~\cite{barbensi2022hypergraphs}. 

\subsection{Optimal transport on persistence diagrams}\label{sec:optimal_transport_persistence_diagrams}

Given a Polish space $X$, let $\mathcal{P}(X)$ denote the set of probability measures supported on $X$. A \emph{coupling} between two distributions $\mu, \mu' \in \mathcal{P}(X)$ is a measure $\pi \in \mathcal{P}(X \times X)$ with \emph{left} and \emph{right} marginals equal to $\mu$ and $\mu'$, respectively. That is,
\[
    \pi(A \times X) = \mu(A) \quad \mbox{and} \quad \pi(X \times A) = \mu'(A)
\]
for all measurable $A \subset X$. We denote the set of all couplings as $\Pi(\mu, \mu')$.

Now let us equip $X$ with a metric $d$. For $p \in [1,\infty)$, the \emph{$p$-Wasserstein distance} \cite[Definition 6.1]{villani2009optimal} between $\mu$ and $\mu'$ is then defined as
\begin{equation}\label{eq:wass}
    d_{\mathrm{W},p}(\mu, \mu'): = \inf_{\pi \in \Pi(\mu,\mu')} \left( \int_{X \times X} d(x,y)^p \: \diff \pi(x, y) \right)^{1/p} .
\end{equation}
This metric is the central object of study in the optimal transportation literature \cite{villani2009optimal, santambrogio2015optimal, peyre2019computational}. It is a standard result that the infimum in \eqref{eq:wass} is realised (i.e. it is in fact a \emph{minimum}) and a minimiser $\pi$ of \eqref{eq:wass} is referred to as an \emph{optimal coupling}. Intuitively, this is a probabilistic matching between points in the supports of $\mu$ and $\mu'$. 

The optimal transport approach of finding matchings can be related to notions of distance on the space of persistence diagrams: the most popular construction is known as the \emph{bottleneck distance}~\cite{cohen2005stability}, which is related to the $\infty$-Wasserstein distance (although, we note that the initial introduction of the bottleneck distance made no reference to optimal transport). Our interest in this article is directed to similar versions of the $p$-Wasserstein distances for $1 \leq p < \infty$. The usual modification of the Wasserstein matching problem to compare persistence diagrams is to allow points to be matched to a \emph{virtual} point $\partial_\Lambda$, representing the diagonal set $\{(a,a) \mid a \in \R\} \subset \R^2$ (i.e. homology classes with vanishing persistence). Formally, for $p \in [1,\infty)$, the \emph{$p$-Wasserstein distance} between two persistence diagrams $D$ and $D'$ (see~\cite[Equation 6]{lacombe2018large}) is defined as
\begin{equation}
    d_{\mathrm{W},p}^{\PD}(D, D') = \min_{\pi \in \Pi (D, D')} \left( \sum_{ (a,b) \in \pi } \| a-b \|^p_p + \sum_{s \in U_\pi} \| s-\Proj_{\partial_\Lambda}(s) \|^p_p \right)^{1/p},\label{eq:partial_matching_persistence_diagrams}
\end{equation}
where the various notations are described as follows:
\begin{itemize}
  \item Slightly abusing notation, to illustrate the analogy to the definition of Wasserstein distance in \eqref{eq:wass}, we denote by $\Pi(D, D')$ the set of all \emph{partial matchings} between points representing homology classes in $D$ and $D'$. For any $\pi \in \Pi(D, D')$, $\pi \subset D \times D'$ has the property that each $a \in D$ and $b \in D'$ appear in at most one ordered pair in $\pi$.
  \item The set $U_\pi$ is the set of \emph{unmatched points} for $\pi$: $s \in U_\pi$ if and only if $s \in D$ and $(s,b) \not \in \pi$ for any $b \in D'$ or $s \in D'$ and $(a,s) \not \in \pi$ for any $a \in D$. 
  \item The map $\Proj_{\partial_\Lambda}$ is the metric projection (w.r.t. any $\|\cdot \|_p$) of $\Lambda$ onto the diagonal, as represented by the point $\partial_\Lambda$. Explicitly, for $s = (x,y)$, 
  \[ \Proj_{\partial_\Lambda}(s) = ((x+y)/2, (x+y)/2). \]
\end{itemize}

The output from solving \eqref{eq:partial_matching_persistence_diagrams} is a partial matching $\pi$ between homology classes in $D$ and $D'$ which is optimal with respect to the cost induced by the $\|\cdot\|_p$-norm in $\mathbb{R}^2$, restricted to $\Lambda$. By construction, this matching only depends on the \emph{coordinates} of points in the persistence diagram. The resulting matching is thus agnostic to aspects of the data that are not captured in the persistence diagram representations. While the generality of the persistence diagram representation may contribute to its wide applicability, in many contexts this may be a limitation. 

Observe that the notation and concepts used to define $d_{\mathrm{W},p}^{\PD}$ are similar to those that appear in the optimal transport framework described above; indeed, the family of persistence diagram distances \eqref{eq:partial_matching_persistence_diagrams} can be generalised via the language of \emph{partial} OT~\cite{figalli2010new, divol2021understanding} by associating measures to the objects under study. We define a \emph{measure persistence diagram} (MPD) to be a Radon measure $\mu$ supported on $\Lambda$. Then, any finite persistence diagram $D = \{(b_i,d_i) \}_{i = 1}^{|D|}$ determines a measure persistence diagram:
\begin{equation}\label{eqn:diagram_to_measure}
    \nu_D = \sum_{i = 1}^{|D|} \delta_{x_i},
\end{equation} 
where the $\delta_{x_i}$ denotes a Dirac supported at $x_i = (b_i , d_i)$. 

Let $\overline{\Lambda}:= \Lambda \cup \{\partial_\Lambda\}$ be obtained by adjoining to $\Lambda$ the virtual point $\partial_\Lambda$, with the resulting space endowed with the disjoint union topology. Given two measure persistence diagrams $\nu$ and $\nu'$, we say a Radon measure $\pi$ on $\overline{\Lambda} \times \overline{\Lambda}$ is \emph{admissible} if it has marginals $\nu$ and $\nu'$ (we consider $\nu$ and $\nu'$ to be measures on $\overline{\Lambda}$ which are supported on $\Lambda$) and it satisfies the additional constraint $\pi((\partial_\Lambda,\partial_\Lambda)) = 0$, and write $\pi \in \Pi_\adm(\nu, \nu')$. Then for $p \in [1,\infty)$, we define the \emph{$p$-Wasserstein distance between measure persistence diagrams}~\cite{divol2021understanding} as:

\begin{equation}\label{eqn:p_Wasserstein_MPD}
    d_{\mathrm{W},p}^{\MPD}(\nu, \nu') = \inf_{\pi \in \Pi_{\rm adm}(\nu, \nu')} \left( \int_{\overline{\Lambda} \times \overline{\Lambda}}  \| x-x' \|^p_p \: \diff \pi(x, x')  \right)^{1/p}.
\end{equation}
If $\nu_D$ and $\nu_{D'}$ arise from finite persistence diagrams $D$ and $D'$, respectively, as in \eqref{eqn:diagram_to_measure}, then it is not hard to show that $d_{\mathrm{W},p}^{\PD}(D,D') = d_{\mathrm{W},p}^{\MPD}(\nu_D,\nu_{D'})$~\cite[Proposition 3.2]{divol2021understanding}. In general, $d_{\mathrm{W},p}^{\MPD}$ is finite, provided that we work in the space of measure persistence diagrams $\mu$ with \emph{finite $p$-persistence}, in the sense that

\begin{equation}\label{eqn:finite_p_persistence}
  \left( \int_\Lambda \|x - \mathrm{Proj}_{\partial_\Lambda}(x)\|^p_p \: \diff\mu(x) \right)^{1/p} < \infty. 
\end{equation}

In what follows, we will restrict to the space of measure persistence diagrams $\mu$ with finite $p$-persistence for every $p \geq 1$. This is mostly for convenience, to avoid the need to include additional technical conditions and notation in the statements of our results. With an abuse of notation, we will still indicate the space of such objects as $\MPD$.

\subsection{Gromov-Wasserstein and co-optimal transport distances}\label{sec:GW_distances}

A \emph{metric measure space (mm-space)}~\cite{sturm2023space, memoli2011gromov} is a triple $M = (X, d, \mu)$, where $X$ is a space endowed with a complete separable metric $d$ and with a fully supported Borel measure $\mu$. Given two mm-spaces $M = (X, d, \mu)$ and $M' = (X', d', \mu')$ and a coupling $\pi \in \Pi(\mu, \mu')$ and $p \in [1,\infty)$, we introduce the $p$-distortion functional $\mathrm{dis}_p^\mathrm{GW}$ \cite{sturm2023space}:
\[
    \disp  = \left( \int_{ (X\times X')^2} | d(x,y) - d'(x',y') |^p \: \diff \pi(x, x') \: \diff \pi(y, y') \right)^{1/p} = \| d -d' \|_{L^p(\pi \otimes \pi)}.
\]
By taking the infimum over all feasible couplings, we obtain the \emph{Gromov-Wasserstein (GW) $p$-distance} between $M$ and $M'$:
\[
    \gw(M,M') := \inf_{\pi \in \Pi(\mu,\mu')} \disp.
\]
This quantity is finite if we work in the space of mm-spaces whose distance functions have finite $p$-th moment; let us make the simplifying convention that we always work with bounded mm-spaces, in order to avoid this technical issue. The GW-distance induces a pseudo-metric on the space of mm-spaces~\cite{memoli2007, memoli2011gromov}, such that $\gw(M,M') = 0$ if and only if there is a measure-preserving isometry from $M$ to $M'$.

For our purposes, we relax the requirement on the function $d : X \times X \to \R$ to be a metric and instead allow any symmetric, measurable and bounded function $k \in L^p_{\mathrm{sym}}(X \times X)$. Then the structure $M = (X, k, \mu)$ is instead referred to as a \emph{gauged measure space}~\cite{sturm2023space} or a \emph{measure network}~\cite{chowdhury2020gromov} (in fact, the definition of a measure network in \cite{chowdhury2020gromov} even drops the symmetry condition). This allows us to consider pairwise relations on $X$ to be modelled by a wider class of functions, such as affinity functions or kernels \cite{chowdhury2021generalized, xu2019gromov}. In this case, the space $(\mathcal{M}, d_{\GW,p})$ of equivalence classes of measure networks (under the equivalence relation $M \sim M' \Leftrightarrow d_{\GW,p}(M,M')=0$) is a complete, geodesic, metric space~\cite{sturm2023space, chowdhury2020gromov}. 

Recently, a GW-like distance has been developed for comparing hypergraph structures~\cite{chowdhury2023hypergraph}, based on ideas from the \emph{co-optimal transport} problem initially introduced in~\cite{titouan2019optimal}. Recall that a hypergraph is defined by a pair $H = (V,E)$, where $V$ is a finite set of vertices and $E$ is a set of hyperedges; each $e \in E$ is a subset of $V$. A general structure which encompasses the notion of a hypergraph is a \emph{measure hypernetwork}, which is a quintuple $H = (X, \mu, Y, \nu, \omega)$, where $(X, \mu)$ and $(Y, \nu)$ are respectively Polish spaces with fully supported Borel probability measures, and $\omega$ is a non-negative, measurable and bounded function $\omega: X \times Y \to \mathbb{R}$. Indeed, any hypergraph $(V,E)$ can be encoded as a measure network $(V,\mu,E,\nu,\omega)$, by choosing $\mu$ and $\nu$ to be uniform probability measures on $V$ and $E$ respectively, and $\omega:V \times E \to \R$ to be an indicator $\omega(v,e) = \ones_{v \in e}$. 

Given two measure hypernetworks $H = (X, \mu, Y, \nu, \omega)$ and $H' = (X', \mu', Y', \nu', \omega')$, the \emph{$p$-co-optimal distortion} for a pair of couplings $(\pi^v, \pi^e) \in \Pi(\mu,\mu') \times \Pi(\nu, \nu')$ (the notation here intended to evoke the idea that $\pi^v$ is a coupling of vertices and $\pi^e$ is a coupling of hyperedges) is defined to be  
\begin{align*}
  \discp(\pi^v, \pi^e) &:= \left( \int_{X\times X' \times Y \times Y'} | \omega(x,y) -\omega'(x',y') |^p \: \diff\pi^v(x, x') \: \diff\pi^e(y, y') \right)^{1/p} \\
  &\:= \| \omega - \omega' \|_{L^p(\pi^v \otimes \pi^e)}.
\end{align*}
Similarly to the GW distance, the \emph{hypernetwork $p$-co-optimal transport distance} is then given by 
\begin{align}
    \coot(H,H'):= & \inf_{\pi^v \in \Pi(\mu,\mu'), \: \pi^e \in \Pi(\nu,\nu')} \discp(\pi^v, \pi^e). \label{eq:cooptimal_transport_definition}
    \end{align}
The expression $\coot(H,H')$ defines a pseudo-metric on the space of measure hypernetworks, and induces a metric on the space of equivalence classes of measure hypernetworks up to isomorphism, which is also shown to be geodesic and complete in~\cite[Theorem 1]{chowdhury2023hypergraph}. 

\section{Topological Optimal Transport}~\label{sec:tpot}

Having discussed matching-based comparison approaches for persistence diagrams, point clouds, and hypergraphs respectively, we turn to the aim of this article, which is to develop a framework for matching point clouds that \emph{couples} geometric and topological information via these structures. Our starting point is to represent a point cloud $X$ together with its topological features as a PH-hypergraph $H = H(X, \mathcal{K}, g)$. Given another point cloud and PH-hypergraph, our objective is to find a \emph{coupling} between them that minimises topological distortion and optimally preserves topological features, in some sense. A na\"ive attempt to a find a topology-informed matching between point clouds $X$ and $X'$ is to simply apply the HyperCOT framework~\cite{chowdhury2023hypergraph} on their PH-hypergraphs $H = H(X, \mathcal{K}, g)$ and $H' = H'(X', \mathcal{K}', g')$. However, this simplistic approach suffers from at least two problems which we list below. 
\begin{itemize}
    \item \textbf{Problem 1.} PH-hypergraphs do not contain any information on the \emph{significance} (measured by persistence) of homology classes. One could naturally weight hyperedges by the persistence value of their corresponding homology classes~\cite{barbensi2022hypergraphs}. However, in an optimal transport setting, the HyperCOT would promote ``splitting'' of mass from hyperedges with higher persistence to hyperedges with lower persistence, rather than the desired effect of matching significant hyperedges. 
    \item \textbf{Problem 2.} PH-hypergraphs can be \emph{disconnected}, and there might be points which do not belong to any hyperedge~\cite{barbensi2022hypergraphs}. In this context, the desired property is to have geometric (spatial) information inform the transport plan. 
\end{itemize}
 The goal of this section is to develop a framework for geometric cycle matching which addresses these problems.

\subsection{Measure topological networks} 

To solve the problems described above, we first introduce an appropriate general model for the structures that we wish to compare. We first introduce necessary notation as follows: for a locally compact Polish metric space $Y$, we denote by $\overline{Y}$ the \emph{augmented} space obtained as the disjoint union $\overline{Y} = Y \cup \{\partial_Y\}$ (with the disjoint union topology), where $\partial_Y$ is an abstract point.

\begin{definition}[Measure Topological Network]\label{def:measure_topological_network}
    A \emph{measure topological network}, or simply \emph{topological network}, is a triple $P = \left((X, k, \mu), (Y, \iota, \nu), \omega\right)$, 
    where 
    \begin{itemize}
        \item $(X, k, \mu)$ is a gauged measure space (see Section \ref{sec:GW_distances}),
        \item $Y$ is a locally compact Polish 
    space, 
        \item $\nu$ is a Radon measure supported on $Y$,  
        \item $\iota:Y \longrightarrow \Lambda$ is a continuous function,
        \item $\iota_{\#}\nu$ is a measure persistence diagram (see Section \ref{sec:optimal_transport_persistence_diagrams}), and
        \item $\omega$ is a measurable and bounded function $\omega: X \times Y \to \mathbb{R}$, so that the quintuple $\left( X, \mu, Y, \nu, \omega  \right)$ is a measure hypernetwork (see Section \ref{sec:GW_distances}). 
    \end{itemize} 
    Let us denote by $\mathcal{P}$ the class of all measure topological networks, with the simplifying conventions that gauged measure spaces have bounded kernels and that measure persistence diagrams have finite $p$-persistence for all $p \geq 1$ (see \eqref{eqn:finite_p_persistence}). 
\end{definition}

\begin{example}[Topological Network from a Metric Measure Space]\label{ex:mmspace_to_top_network}
  A measure topological network arises from the data of a PH-hypergraph defined over a metric measure space, and this is inspiration for the definition and the source of all computational examples. Indeed, given a finite metric measure space $(X,d,\mu)$, an associated persistence diagram $D$ (obtained via, say, Vietoris-Rips persistent homology) and a choice of a set of generating cycles $g$ for $D$, we construct an associated topological network as follows. First, we take $Y = D$ as a multiset of points in $\Lambda$ (one can consider this as a proper set by indexing its points, so that $Y$ can be considered a finite topological space). Taking $\nu$ to be the uniform (i.e., counting) measure and $\iota:Y \to \Lambda$ to be inclusion, we have that $\iota_{\#}\nu = \nu_D$ as in \eqref{eqn:diagram_to_measure}.  Finally, $\omega$ is a kernel representing the hypergraph structure $(X,g)$.

  The simplest choice is to use a binary incidence function: for $x \in X$ and a point $(a,b) \in D$,  represented by a cycle $c \in g$, $\omega(x,c) = 1$ if and only if $x$ is a vertex of any simplex in $c$.
  The flexibility in the definition allows for other hypergraph kernels which may have more desirable properties in practice, such as those studied in~\cite{chowdhury2023hypergraph}.
  Of note, the recently developed method TOPF \cite{grande2024node} produces node-level representations of topological features in the form of a scalar function on points and topological features, and as such fits within our framework. 
  Finally, we can replace the metric $d$ with a symmetric gauge function $k$ that may better encode local geometry. We refer the reader to Section \ref{sec: results} for examples of different choices of incidence and gauge functions. 
\end{example} 

\begin{remark}\label{rmk:notopo}
    In full generality, measure topological networks are not restricted to describing topological features of the underlying space. The construction described in Example~\ref{ex:mmspace_to_top_network} motivates the use of ``topological'' in the name.
\end{remark}

Although computational examples of topological networks in this paper will always be constructed as in Example \ref{ex:mmspace_to_top_network}, we provide the following additional construction to motivate the level of generality of our definition.

\begin{example}[Topological Network from a Curvature Set]\label{ex:curvature_set}
Let $(X,d,\mu)$ be a finite metric space. For fixed $k < |X|$, consider the map $d_k:X^k \to \R^{k \times k}$ defined by $d_k(x_1,\ldots,x_k) = (d(x_i,x_j))_{i,j=1}^k$; that is, this map takes a $k$-tuple of points to its distance matrix. The image of this map is an invariant of $X$ introduced by Gromov in~\cite{gromov1999metric}, called the $k$th \emph{curvature set} of $X$. Given a $k \times k$ distance matrix, one can apply degree-$\ell$ Vietoris-Rips persistent homology, so that the composition yields a map $D_{k,\ell}$ from $X^k$ into the space of persistence diagrams---this invariant and other related invariants were thoroughly studied in the recent paper~\cite{gomez2024curvature}. This structure gives rise to a topological network as follows. Let $\widetilde{Y}$ be the multiset consisting of all persistence points arising in diagrams in the image of $D_{k,\ell}$, let $Y$ be its underlying set, let $\nu$ be the pushforward of uniform measure on $\widetilde{Y}$ to $Y$ (so that $\nu$ counts persistence points with multiplicity), and let $\iota:Y \to \Lambda$ be the inclusion map. We then define $\omega:X \times Y \to \R$ as 
\[
\omega(x,y) = \mathbb{P}_{\mu^{\otimes (k-1)}} \big( y \in D_{k,\ell}(x,x_1,\ldots,x_{k-1})\big).
\]
\end{example}

\begin{remark}[Finite and Infinite Diagrams]
    We acknowledge that our main examples of topological networks from Example \ref{ex:mmspace_to_top_network} (as well as the examples described in Example \ref{ex:curvature_set}) have the property that $Y$ is finite; that is, the relevant MPDs are finitely supported. One could safely work under this assumption throughout the analysis in the rest of the paper. However, this assumption doesn't make any of the analysis significantly easier, so we opt to work at the level of generality which allows for sets $Y$ with infinite cardinality. Moreover, there are interesting examples where non-finitely supported MPDs arise, and we prefer to work with a formalism which is able to handle these examples, with a view toward future applications. One source of such examples is the statistical theory of the topology of random point clouds, where non-finite MPDs play an important role as asymptotic distributions when the number of points tends to infinity \cite{chazal2018density,divol2019choice,goel2018asymptotic}. \\
\end{remark}

We now aim to define a notion of distance between measure topological networks, which will be the main object of study throughout the rest of the paper. First, we synthetically extend the definition of an admissible coupling from the setting of persistence diagrams (see Section \ref{sec:optimal_transport_persistence_diagrams}) to the general setting of measure spaces with adjoined virtual points.

\begin{definition}[Admissible Couplings]
    Let $(Y,\nu)$ and $(Y',\nu')$ be a locally compact Polish spaces endowed with a Radon measures and let $\overline{Y} = Y \cup \{\partial_Y\}$ and $\overline{Y'} = Y' \cup \{\partial_{Y'}\}$ be the associated augmented spaces. We say that a Radon measure $\pi$ on $\overline{Y} \times \overline{Y'}$ is \emph{admissible} if it has marginals $\nu$ and $\nu'$ and satisfies $\pi((\partial_Y,\partial_{Y'})) = 0$. We denote the set of admissible couplings by $\Pi_{\mathrm{adm}}(\nu,\nu')$.
\end{definition}

This leads to the following main definition.

\begin{definition}[Topological Optimal Transport (TpOT)] \label{def:TpOT}
    Given two topological networks $P = \left((X, k, \mu), (Y, \iota, \nu), \omega\right)$ and $P' =  \left( (X', k', \mu'), (Y', \iota', \nu'),  \omega' \right)$, we formulate the \emph{Topological Optimal Transport (TpOT) problem} as: 
    \begin{align}
          d_{\TpOT, p}(P, P') &:= \inf_{\substack{\pi^v \in \Pi(\mu,\mu') \\ \pi^e \in \Pi_{\rm adm}(\nu,\nu')}} \left(  \int_{\overline{Y} \times \overline{Y'}} \| \iota(y) - \iota'(y') \|_p^p \: \diff \pi^e(y, y') \label{eqn:TpOT1}  \right. \\
                              & \quad\quad\quad  +  \int_{(X\times X')^2} | k(x,y) -k'(x',y') |^p \: \diff\pi^v(x, x') \: \diff\pi^v(y, y')  \label{eqn:TpOT2} \\ 
                              & \quad\quad\quad  \left. + \int_{X\times X'\times \overline{Y} \times \overline{Y'}} | \omega(x,y) - \omega'(x',y') |^p \: \diff\pi^v(x, x') \: \diff\pi^e(y, y')\right)^{\frac{1}{p}}.\label{eqn:TpOT3}
\end{align} 
Letting $d_p$ denote the $\ell^p$-distance on $\R^2$, this can be expressed more concisely as
\begin{equation}\label{eq:tpotCON}
\begin{split}
  d_{\TpOT, p}(P, P') &= \inf_{\pi^v, \pi^e} \left(  \| d_p \circ (\iota, \iota') \ \|^p_{L^p(\pi^e)}\right. \\
  &  \qquad \qquad \qquad  + \left.\| k - k' \|^p_{L^p(\pi^v \otimes \pi^v)}
  +\| \omega - \omega' \|^p_{L^p(\pi^v \otimes \pi^e)} \right)^{\frac{1}{p}}.
\end{split}
\end{equation}
\end{definition}

Here we are implicitly extending $\omega$ and $\iota$ to functions $\omega:  X \times \overline{Y} \to \mathbb{R}$, $\iota:  \overline{Y} \to \overline{\Lambda}$ by setting $\omega\lvert_{X \times \partial_Y} = 0$ and $\iota(\partial_Y) = \partial_\Lambda$. The restriction $\omega\lvert_{X \times \partial_Y}$ can be interpreted as (trivial) a membership function (via the map $\iota$) for the diagonal $\partial_\Lambda$. 

Let us interpret the various components of this definition, when the measure topological spaces arise from PH-hypergraphs, as described in Example \ref{ex:mmspace_to_top_network}.
\begin{itemize}
    \item The coupling $\pi^v$ defines a probabilistic matching between the points of the gauged measure spaces (see Figure~\ref{fig:tpot_pipeline}, left column) and the admissible coupling $\pi^e$ defines a probabilistic matching between generating cycles of the persistence diagrams (see Figure~\ref{fig:tpot_pipeline}, middle column).
    \item The integral in \eqref{eqn:TpOT1} is a partial transport term which measures the quality of the matching of generating cycles. This term is intended to address Problem 1.
    \item The integral in \eqref{eqn:TpOT2} is a ``Gromov-Wasserstein'' term which matches the quality of the matching of points in the gauged measure spaces. This term is intended to address Problem 2.
    \item The integral in \eqref{eqn:TpOT3} is a ``co-optimal transport'' term which addresses the basic problem of matching PH-hypergraphs.
\end{itemize}

\begin{remark}\label{rem:weights}
    In applications, it is frequently useful to include tunable weights on the terms of \eqref{eq:tpotCON}, which can be used as hyperparameters to accentuate topology or geometry as is appropriate for a given task. To keep the exposition clean, we suppress these weights from theoretical considerations. They are explicitly included in the exposition of the computational pipeline---see Section \ref{sec:discrete}.
\end{remark}

\subsection{Metric properties of TpOT}\label{sec:metric}

Next we will show that $d_{\rm TpOT, p}$ defines a metric on the space of (certain equivalence classes of) topological networks. As a first step, we have the following proposition. 

\begin{proposition}\label{prop:realised}
The infimum of Equation~\eqref{eq:tpotCON} is realised.
\end{proposition}

\begin{proof}[Proof of Proposition \ref{prop:realised}]
It is a classic result that the coupling space $\Pi(\mu, \mu')$ is (sequentially) compact in $P(X \times X')$, where $P(X \times X')$ is the set of Borel probability measures on $X \times X'$, see, \textit{e.g.},~\cite[Lemma 4.4]{villani2009optimal} and~\cite[proof of Lemma 24]{chowdhury2023hypergraph}, where we topologise spaces of measures with the weak topology. Similarly, $\Pi_{\rm adm}(\nu, \nu')$ is sequentially compact in the space $M(\overline{Y} \times \overline{Y'})$ of Radon measures on $\overline{Y} \times \overline{Y'}$ (see~\cite[proof of Prop 3.1]{divol2021understanding}, which extends to Polish spaces as in Definition~\ref{def:measure_topological_network}) with the vague topology, i.e. given by duality between measures and continuous functions with compact support. Together, these imply that $\Pi(\mu, \mu') \times \Pi_{\rm adm}(\nu, \nu')$ is sequentially compact in $P(X \times X') \times M(\overline{Y} \times \overline{Y'})$ with the product topology. 

To complete the proof, we just need to show lower-semicontinuity of the function $C_p: \Pi(\mu, \mu') \times \Pi_{\rm adm}(\nu, \nu') \rightarrow  \R$ defined by
\begin{equation}
\begin{split}
  C_p(\pi, \xi) &= \ \int_{X\times X'\times \overline{Y} \times \overline{Y'}} | \omega - \omega' |^p \: \diff(\pi \otimes \xi) \\
  &\qquad \qquad + \int_{(X\times X')^2} | k - k' |^p \: \diff (\pi \otimes \pi) + \int_{\overline{Y} \times \overline{Y'}}  | d_p\circ(\iota,\iota') |^p \: \diff \xi.
  \end{split}
\end{equation}\label{eq:continuity}
Continuity of each of the second two terms follows exactly as in the proofs of~\cite[Lemma 24]{chowdhury2023hypergraph} and~\cite[Proposition 3.1]{divol2021understanding}, respectively. For the first term, we consider a weakly converging $\pi_n^v \overset{\text{w}} \longrightarrow \pi^v$, and a vaguely converging $\pi_n^e \overset{\text{v}} \longrightarrow \pi^e$. Then, by the same argument as in~\cite[Proposition 3.1]{divol2021understanding}, the sequence $(\pi^v \otimes \pi^e)_n' : U \mapsto \int_{U} | \omega - \omega' |^p \: \diff (\pi^v_n \otimes \pi^e_n)$ converges to $(\pi^v \otimes \pi^e)' : U \mapsto \int_{U} | \omega - \omega' |^p \: \diff(\pi^v \otimes \pi^e)$. By using the Portmanteau Theorem (\cite[Proposition A.4]{divol2021understanding}, which also holds for weak convergence) as in~\cite[Proposition 3.1]{divol2021understanding}, continuity follows.
\end{proof}

\begin{proposition}\label{prop:psemet}
The function $d_{\mathrm{TpOT}, p}$ defines a pseudo-metric on the space of topological networks $\mathcal{P}$. 
\end{proposition}

To prove Proposition~\ref{prop:psemet} we only need to show that the triangle inequality holds, as symmetry is obvious from the definition of $d_{\mathrm{TpOT}, p}$. We follow the usual approach, which relies on the following standard result (see \textit{e.g.}~\cite[Lemma 1.4]{sturm2023space}).

\begin{lemma}[Gluing Lemma]\label{lem:gluing}
Let $(X_i, \mu_i), i = 1, \cdots, n$, be Polish probability spaces. Consider couplings $ \xi_i \in \Pi(\mu_{i-1}, \mu_{i})$, for $i = 2, \cdots, n $. There exists a unique probability measure $\xi$ on $X_1 \times \cdots \times X_n$  such that $(p_{i-1} \times p_{i})_\# \xi = \xi_i$ for every $ i = 2, \cdots, n$, where $p_i : X_1 \times \cdots \times X_n \longrightarrow X_i $ is the projection on the $i$th factor.  
\end{lemma}

The unique measure guaranteed by the lemma will be referred to as the \emph{gluing} of $\xi_2,\ldots,\xi_n$ and will sometimes be denoted as $\xi_2 \boxtimes \xi_3 \boxtimes \cdots \boxtimes \xi_n$. Lemma~\ref{lem:gluing} can be adapted to the case of Radon measures and admissible couplings~\cite[Proposition~3.2]{divol2021understanding}.  

\begin{proof}[Proof of Proposition~\ref{prop:psemet}]
    The proof follows adapts techniques used in the proof of~\cite[Proposition~3.2]{divol2021understanding} and~\cite[Theorem~1]{chowdhury2023hypergraph}. To prove that the triangle inequality holds for $d_{\mathrm{TpOT}, p}$, consider three topological networks $P_1$, $P_2$ and $P_3$, and let $(\pi_{12}^v, \pi^e_{12})$ and $(\pi_{23}^v, \pi^e_{23})$ be the couplings realising $d_{\mathrm{TpOT}, p}(P_1, P_2)$ and $d_{\mathrm{TpOT}, p}(P_2, P_3)$, which exist by Prop.~\ref{prop:realised}. Then, using the gluing lemma, we construct a probability measure $\pi^v$ on $X_1 \times X_2 \times X_3$ with marginals $\pi^v_{12}$ and $\pi^v_{23}$ on $X_1 \times X_2$ and $X_2 \times X_3$, respectively. We denote by $\pi^v_{13}$ the marginal on $X_1 \times X_3$. Similarly, the gluing lemma adapted to Radon measures and admissible couplings yields a measure $\pi^e$ on  $\overline{ Y_1} \times \overline{ Y_2} \times \overline{ Y_3} $ with marginals that agree with $\pi^e_{ij}$ when restricted to $\overline{Y_i} \times \overline{Y_j} \setminus \{(\partial_{Y_i},\partial_{Y_j})\}$ and that induces a zero cost on $ (\partial_{Y_i},\partial_{Y_j})$. Then, we have

\begin{align}
      &d_{\TpOT, p}(P_1, P_3) \nonumber \\
      & \leq  \left(\| d_p \circ(\iota_1,\iota_3) \|^p_{L^p(\pi^e_{13})} + \| k_1 - k_3 \|^p_{L^p(\pi^v_{13} \otimes \pi^v_{13})} +  \| \omega_1 - \omega_3 \|^p_{L^p(\pi^v_{13} \otimes \pi^e_{13})} \right)^{1/p}\label{eq:triangle_ineq1} \\
                               &= \left(\|  d_p  \circ(\iota_1,\iota_3)  \|^p_{L^p(\pi^e)} +  \|   k_1 - k_3 \|^p_{L^p(\pi^v \otimes \pi^v)} +  \| \omega_1 - \omega_3 \|^p_{L^p(\pi^v \otimes \pi^e)}\right)^{1/p} \label{eq:triangle_ineq2} \\
                               & = \left\| \begin{pmatrix}
                                 \|  d_p  \circ(\iota_1,\iota_3)  \|_{L^p(\pi^e)}   \\
                                 \|   k_1 - k_3 \|_{L^p(\pi^v \otimes \pi^v)} \\
                                 \| \omega_1 - \omega_3 \|_{L^p(\pi^v \otimes \pi^e)}
                               \end{pmatrix} \right\|_p \nonumber \\
                               & \leq  \left\| \begin{pmatrix}
                                 \|  d_p  \circ(\iota_1,\iota_2)  \|_{L^p(\pi^e)}  +  \|  d_p \circ(\iota_2,\iota_3) \|_{L^p(\pi^e)} \\
                                 \|   k_1 - k_2  \|_{L^p(\pi^v \otimes \pi^v)} +  \|  k_2 - k_3 \|_{L^p(\pi^v \otimes \pi^v)} \\
                                 \| \omega_1 - \omega_2 \|_{L^p(\pi^v \otimes \pi^e)} +  \| \omega_2 - \omega_3 \|_{L^p(\pi^v \otimes \pi^e)}
                               \end{pmatrix} \right\|_p    \label{eq:triangle_ineq3} \\        
                               & \leq         \left\| \begin{pmatrix}
                                 \|  d_p  \circ(\iota_1,\iota_2)  \|_{L^p(\pi^e)}   \\
                                 \|   k_1 - k_2  \|_{L^p(\pi^v \otimes \pi^v)}  \\
                                 \| \omega_1 - \omega_2 \|_{L^p(\pi^v \otimes \pi^e)}
                               \end{pmatrix} \right\|_p  +  \left\| \begin{pmatrix}
                                 \|  d_p  \circ(\iota_2,\iota_3)  \|_{L^p(\pi^e)} \\
                                 \|  k_2 - k_3 \|_{L^p(\pi^v \otimes \pi^v)} \\
                                 \| \omega_2 - \omega_3 \|_{L^p(\pi^v \otimes \pi^e)}
                               \end{pmatrix} \right\|_p  \label{eq:triangle_ineq4}\\   
                               &   =    \left\| \begin{pmatrix}
                                 \|  d_p \circ(\iota_1,\iota_2)  \|_{L^p(\pi^e_{12})}   \\
                                 \|   k_1 - k_2  \|_{L^p(\pi^v_{12} \otimes \pi^v_{12})}  \\
                                 \| \omega_1 - \omega_2 \|_{L^p(\pi^v_{12} \otimes \pi^e_{12})}
                               \end{pmatrix} \right\|_p  +  \left\| \begin{pmatrix}
                                 \|  d_p \circ(\iota_2,\iota_3)  \|_{L^p(\pi^e_{23})} \\
                                 \|  k_2 - k_3 \|_{L^p(\pi^v_{23} \otimes \pi^v_{23})} \\
                                 \| \omega_2 - \omega_3 \|_{L^p(\pi^v_{23} \otimes \pi^e_{23})}
                               \end{pmatrix} \right\|_p \label{eq:triangle_ineq5} \\ 
                               & = d_{\TpOT, p}(P_1, P_2) + d_{\TpOT, p}(P_2, P_3). \nonumber
\end{align}
    Here, we use $\|\cdot\|_p$ to indicate the standard $\ell^p$ norm on $\mathbb{R}^3$. The inequality \eqref{eq:triangle_ineq1} follows by sub-optimality, while \eqref{eq:triangle_ineq2} and \eqref{eq:triangle_ineq5} by marginal conditions. The triangle inequalities in $L^p$ and $d_p$ give \eqref{eq:triangle_ineq3}, and the triangle inequality in $\ell^p$ gives \eqref{eq:triangle_ineq4}. 
\end{proof}

In analogy with measure (hyper)networks~\cite{chowdhury2019gromov, chowdhury2023hypergraph}, we define an equivalence class on $\mathcal{P}$ as follows.

\begin{definition}[Weak Isomorphism]
  A \emph{weak isomorphism} of topological networks $P, P'$ is a pair of measure-preserving maps $\phi: X \longrightarrow X' $ and $\psi: Y \longrightarrow Y'$ such that $\iota(y) = \iota'(\psi(y))$ for $\nu$-almost every $y \in Y$, $\omega(x,y) = \omega'(\phi(x), \psi(y))$ for $\mu \otimes \nu $-almost every $(x,y) \in X \times Y$, and $k(x_1,x_2) = k'(\phi(x_1), \phi(x_2))$ for $\mu \otimes \mu $-almost every $(x_1,x_2) \in X \times X$. We say that $P, P'$ are \emph{weakly isomorphic} if there exists a topological network $\overline{P}$ and weak isomorphisms from $\overline{P}$ to $P$ and to $P'$. We use $P \sim_w P'$ to denote that $P$ and $P'$ are weakly isomorphic; this is an equivalence relation on $\mathcal{P}$. 
\end{definition}

\begin{proposition}\label{prop:weakiso}
  Two measure topological networks $P = \left((X, k, \mu), (Y, \iota, \nu), \omega\right)$ and $P' =  \left( (X', k', \mu'), (Y', \iota', \nu'),  \omega' \right)$ have distance $ d_{\TpOT, p}(P, P') = 0$ if and only if $P \sim_w P'$.
  \label{prop:weakly_isomorphic}
\end{proposition}

\begin{proof}
  The ``if'' direction is clear, by the triangle inequality. Conversely, assume $ d_{\TpOT, p}(P, P') = 0$ and let $\pi^v $ and $\pi^e$ be couplings realising this distance (Proposition~\ref{prop:realised}). 

We construct $\Tilde{P}$ by setting $\Tilde{X} = X \times X'$, $\Tilde{\mu} = \pi^v$. Similarly, we take 
$\Tilde{Y} = \left( (Y \times Y') \cup (Y \times \partial_{Y'}) \cup (\partial_{Y} \times Y')  \right)$ (where we write $Y \times \partial_{Y'}$ rather than $Y \times \{\partial_{Y'}\}$ for the sake of cleaner notation, and take similar conventions elsewhere), augmented to 
$ \overline{Y} \times \overline{Y'} =  \Tilde{Y} \cup \partial_{Y} \times \partial_{Y'}$, and with measure $\Tilde{\nu} = \pi^e$. Next, we define $\phi,\phi'$ and $\psi,\psi'$ as coordinate projection maps.

  Given $\Tilde{x}_1 = (x_1, x_1'), \Tilde{x}_2 = (x_2, x_2') \in X \times X'$, we set $\Tilde{k}(\Tilde{x}_1, \Tilde{x}_2)  = k(\phi(\Tilde{x}_1), \phi(\Tilde{x}_2)) = k(x_1, x_2)$. By optimality of $\pi^v$ and since $d_{\TpOT, p}(P, P') = 0$, we have that $k(\phi(\Tilde{x}_1), \phi(\Tilde{x}_2)) = k'(\phi'(\Tilde{x}_1), \phi'(\Tilde{x}_2)) $ for almost every $\Tilde{x}_1,\Tilde{x}_1$.
  
 Similarly, for $\Tilde{y}_1 = (y_1, y_1'), \Tilde{y}_2 = (y_2, y_2') \in Y \times Y'$ if we define $\Tilde{\omega}(\Tilde{x}, \Tilde{y}) = \omega(\phi(x), \psi(y))$, by optimality of $\pi^v$ we have $\Tilde{\omega}(\Tilde{x}, \Tilde{y}) = \omega(\phi(x), \psi(y)) = \omega'(\phi'(x), \psi'(y))$ for $\pi^v \otimes \pi^e$-almost every $(\Tilde{x}, \Tilde{y}) \in \Tilde{X} \times \Tilde{Y}$.
  Finally, we set $\Tilde{\iota}: \Tilde{Y} \longrightarrow \Delta$ as $\Tilde{\iota}(\Tilde{y}) = \iota(\psi(y)) = \iota'(\psi'(y'))$ for almost every $\Tilde{y}$, again by optimality of $\pi^e$. This implies that $P$ and $P'$ are weakly isomorphic. 
\end{proof}
The work above immediately implies the following.
\begin{theorem}\label{thm:metric}
The pseudometric $d_{\TpOT, p}$ induces a metric on the quotient space $ \faktor{\mathcal{P}}{\sim_w}$.
\end{theorem}

We abuse notation and also denote by $d_{\TpOT,p}$ the induced metric on the quotient space.

\subsection{A topology-driven metric on point clouds} \label{sec:discrete}
We now summarise the metric described above in the discrete setting of point clouds with finite number of points. Consider a point cloud $X = \{ x_1, \cdots, x_N \} \subseteq \mathbb{R}^n$. Any choice of filtration $\mathcal{K}$ over $X$ yields a corresponding persistence diagram $D = D(X,\mathcal{K})$. Choosing a representative cycle for each homology class in $D$ yields the set of generators $g$, and the corresponding PH-hypergraph $H = H(X, \mathcal{K}, g )$. We can now endow $X$ with a measure topological hypergraph structure by considering 
\[
P_X =  \left( (X, k, \mu_{X}), (Y,\iota,\nu),  \omega_H   \right),
\]
where $k$ is a gauge function of choice defined on $X$ (for instance, a kernel function, or pairwise distances), $\mu_X$ is a chosen measure on $X$ (for instance, uniform), $Y$ is a set with $|D|$ elements (counted with multiplicity), $\nu$ and $\iota$ are such that $ \iota_{\#} \nu = \nu_D$, where $\nu_D$ is as in \eqref{eqn:diagram_to_measure}, and $\omega_H$ is the binary incidence function for $H$ (see Example~\ref{ex:mmspace_to_top_network}).

Generating cycles are far from being unique, and there are currently several different algorithms and software available to compute them~\cite{barbensi2022hypergraphs,ChenGenerators, vcufar2020ripserer,vcufar2021fast,emmett2015multiscale,  Dey2019, LiMinimalCycle, Obayashi2018VolumeOC, Ripser_involuted}. As explained in Section~\ref{sec: results}, here persistent homology computations are performed using the Julia software \texttt{Ripserer.jl}~\cite{vcufar2020ripserer}, which implements the \texttt{involutive} algorithm~\cite{vcufar2020ripserer,vcufar2021fast} to compute homology and representatives.

Given two point clouds $X = \{ x_1, \cdots, x_N \} \subseteq \mathbb{R}^n$ and $X' = \{ x'_1, \cdots, x'_{N'} \} \subseteq \mathbb{R}^{n'}$ endowed with filtrations, let $P_X$ and $P_{X'}$ be their associated topological networks. In this concrete setting, we focus on the $p=2$ version of the metric, since it is the case where computation of terms quadratic in the couplings can be done efficiently \cite[Proposition 1]{peyre2016gromov}. As was described in Remark \ref{rem:weights}, it is useful in applications to include tunable weights on each of the terms of TpOT. Using notation similar to that of~\cite{chowdhury2020gromov}, the distance $d_{\TpOT,2} = d_{\TpOT, \alpha, \beta}$ can be reformulated as 
\begin{align}\label{eq:tpotdiscrete}
  \begin{split}
    d_{\TpOT, \alpha, \beta}(P_X, P_{X'}) = \min_{\pi^v \in \Pi(\mu, \mu'), \: \pi^e \in \Pi(\tilde{\nu}, \tilde{\nu}')} \quad \Biggl( &\alpha \left\langle L(C, C'), \pi^v \otimes \pi^v \right\rangle  \\
                                                                                                                           &+ (1 - \alpha) \left\langle \tilde{C}(D,D'), \pi^e \right\rangle \\
                                                                                                                           &+ \beta \left\langle L(\omega_H, \omega_{H'}), \pi^v \otimes \pi^e \right \rangle \Biggr)^{1/2}.
  \end{split}
\end{align}
The notations used here are as follows:
\begin{itemize}
    \item All inner products, denoted by angle brackets, in this expression are Frobenius inner products of matrices.
    \item $C \in \mathbb{R}^{|X| \times |X|}, C' \in \mathbb{R}^{|X'| \times |X'|}$ are pairwise affinity matrices for points in $X$ and $X'$ respectively, and $D, D'$ are the persistence diagrams associated to the given filtrations $\mathcal{K}(X)$ and $\mathcal{K}(X')$.
    \item $\tilde{C}(D,D') \in \mathbb{R}^{(|D|+1) \times (|D'|+1)}$ denotes the \emph{augmented} squared Euclidean cost matrix for persistence diagrams, whose last column and row correspond to transport to and from $\partial_\Lambda$ respectively, following the diagonal projection $\pi_{\partial_\Lambda}$ (see~\cite[Equation 8]{lacombe2018large}).
    \item $L(C, C') \in \mathbb{R}^{|X| \times |X'| \times |X| \times |X'|}$ corresponds to the coefficients of the Gromov-Wasserstein distortion functional, given by $L(C, C')_{ijk\ell} =  \frac{1}{2} |C_{ik} - C'_{j\ell}|^2$. 
    \item $\omega_H \in \mathbb{R}^{|X| \times (|D|+1)}$ and $\omega_{H'} \in \mathbb{R}^{|X'| \times (|D'|+1)}$ the hypernetwork functions of the PH-hypergraphs $H$ and $H'$, respectively. In particular, we include an additional last column of zeros representing the diagonal $\partial_\Lambda$.
    \item $L(\omega_H, \omega_{H'})$ corresponds to the coefficients of the co-optimal transport distortion functional, and is given by:
        \begin{align}
            L(\omega, \omega')_{ijkl} = \frac{1}{2} \begin{cases}
                | \omega_{ik} - \omega'_{jl} |^2, &\text{ for } 1 \leq k \leq |D|, 1 \leq l \leq |D'|, \\ 
                |\omega_{ik}|^2, &\text{ for } 1 \leq k \leq |D|, l = |D'|+1\\
                |\omega'_{jl}|^2, &\text{ for } k = |D|+1, 1 \leq l \leq |D'|\\ 
                0, &\text{ for } k = |D|+1, l = |D'|+1.
            \end{cases}
        \end{align}
    In other words, the cost of matching any ``real'' (i.e., not the virtual one introduced at $|D|+1$) edge to the diagonal (representing $\partial_\Lambda$) is given by the squared $L^2$-norm of the edge, while matching the diagonal with the diagonal has zero cost. 
    \item The parameter $\alpha \in [0,1]$ controls the tradeoff between the cost of matchings in the persistence diagram space (coefficient $1-\alpha$) and matchings in terms of the Gromov-Wasserstein distortion (coefficient $\alpha$).
    
    \item The parameter $\beta \in [0, \infty)$ controls the degree to which the geometric and topological matchings are coupled by the hypergraph structure.
\end{itemize}

The solution of the TpOT problem in \eqref{eq:tpotdiscrete} involves determining an optimal pair of couplings $(\pi^v, \pi^e)$. The coupling $\pi^v$ induces a transport plan between points in the point clouds $X,X'$. When $\beta>0$ and $0 < \alpha < 1$, this transport balances between preserving topological features and pairwise distances. The coupling $\pi^e$ induces a partial matching between homology classes in the persistence diagrams $D,D'$. For $0< \alpha <1$, this matching is informed by the proximity of points, corresponding to representative cycles of classes in $D, D'$ in the persistence diagram space $\Lambda$. Implementation details are described in Section~\ref{sec: results}, where we also provide numerical examples.

\subsection{Numerical algorithms}\label{sec:solving}

We now aim at numerical algorithms for approximating the solution to the TpOT problem in practice. Starting from \eqref{eq:tpotdiscrete}, we consider a \emph{entropically regularised} variant of the TpOT problem by adding an entropic regularisation to the transport plans $(\pi^e, \pi^v)$. Writing
\begin{equation}
  \Lc(\pi^v, \pi^e) := \alpha \langle L(C, C'), \pi^v \otimes \pi^v \rangle + (1-\alpha) \langle\tilde{C}(D,D'), \pi^e \rangle + \beta \langle L(\omega_H, \omega_{H'}), \pi^v \otimes \pi^e \rangle, 
\end{equation}
we have 
\begin{align}
  \min_{\pi^v \in \Pi(\mu, \mu'),\:\pi^e \in \Pi(\tilde{\nu}, \tilde{\nu}')} & \Lc(\pi^v, \pi^e) + \varepsilon_v \KL(\pi^v | \mu \otimes \mu') + \varepsilon_e \KL(\pi^e | \tilde{\nu} \otimes \tilde{\nu}'),
\end{align}
where $\varepsilon_v, \varepsilon_e > 0$ specify the strength of the entropic regularisation, and allow us to utilise fast, smooth optimisation techniques \cite{redko2020co, peyre2019computational, peyre2016gromov}. In particular, we can find a local minimum by projected gradient descent~\cite[Section 2.3]{peyre2016gromov}, where both the gradient and the projection are calculated with respect to the KL-divergence. This leads to the following iterative scheme:
\begin{align}
  \begin{split}
    \pi^v_{t+1} &\gets \Proj_{\Pi(\mu, \mu')}^{\KL} \left[ e^{-\varepsilon_v^{-1} \nabla_v \Lc(\pi^v_t, \pi^e_t)} (\mu \otimes \mu')\right], \\ 
    \pi^e_{t+1} &\gets \Proj_{\Pi(\tilde{\nu}, \tilde{\nu}')}^{\KL} \left[ e^{-\varepsilon_e^{-1} \nabla_e \Lc(\pi^v_t, \pi^e_t)} (\tilde{\nu} \otimes \tilde{\nu}')\right].
  \end{split}
\end{align}
Each of these projections can be calculated by matrix scaling using the Sinkhorn algorithm~\cite[Chapter 4]{peyre2019computational}{, specifically:
\begin{equation}
    \Proj^{\KL}_{\Pi(\alpha, \alpha')}\left(e^{\frac{-C}{\varepsilon}} (\alpha \otimes \alpha')\right) = \min_{\pi \in \Pi(\alpha, \alpha')} \langle C, \pi \rangle + \varepsilon \KL(\pi | \alpha \otimes \alpha').
\end{equation}}
The gradients of $\Lc$ in $(\pi^v, \pi^e)$ have the following closed-form expressions:
\begin{align}
    \nabla_v \Lc(\pi^v, \pi^e) &= 2\alpha L(C, C') \otimes \pi^v + \beta L(\omega_H, \omega_H') \otimes \pi^e, \label{eq:grad_L_v} \\
    \nabla_e \Lc(\pi^v, \pi^e) &= (1-\alpha) \tilde{C}[D, D'] + \beta L(\omega_H^\top, {\omega_H'}^\top) \otimes \pi^v.  \label{eq:grad_L_e}
\end{align}
In the unregularised case (i.e., when $\varepsilon_v = \varepsilon_e = 0$), an alternating minimisation in $(\pi^v, \pi^e)$ can be used to find a local minimum.
Fixing $\pi^e$, the minimisation problem in $\pi^v \in \Pi(\mu, \mu')$ is 
\begin{align}\label{eq:FGWproblem}
  \min_{\pi^v \in \Pi(\mu, \mu')} \alpha \langle L(C, C'), \pi^v \otimes \pi^v \rangle + \beta  \langle  L(\omega_H, \omega_{H'}) \otimes \pi^e, \pi^v \rangle. 
\end{align}
This falls within the fused Gromov-Wasserstein framework which was introduced and studied in detail by~\cite{titouan2019optimal, vayer2020fused}. A local minimum can be found using a conditional gradient method~\cite[Algorithm 1]{titouan2019optimal}.
On the other hand, fixing $\pi^v$ and minimising in $\pi^e \in \Pi(\tilde{\nu}, \tilde{\nu}')$, we have
\begin{align}\label{eq:coot_solve}
    \min_{\pi^e \in \Pi(\tilde{\nu}, \tilde{\nu}')} \: \langle M , \pi^e \rangle, \quad \text{ where } M = \pi^v \otimes L(\omega_H, \omega_{H'}) + (1-\alpha) \tilde{C}(D, D')
\end{align}
This amounts to a standard optimal transport problem, and we have the identity (see \cite[Proposition 1]{peyre2016gromov})
\[
  \langle \pi \otimes L(X, X'), \xi \rangle = \langle L(X^\top, (X')^\top) \otimes \pi, \xi \rangle  \rangle = \langle -X^\top \pi (X'), \xi \rangle.
\]

For clarity, we provide pseudocode below for the entropy regularised approach (Algorithm \ref{alg:tpot_entropic}) as well as the unregularised approach (Algorithm \ref{alg:tpot_unreg}).

\begin{algorithm}
\caption{Entropy-regularised TpOT: Projected gradient descent}
\begin{algorithmic}[1]
\State \textbf{Input:} Pairwise affinity matrices $C, C'$, persistence diagrams $D, D'$, incidence matrices $\omega, \omega'$, probability distributions $\mu, \mu'$, positive measures $\nu, \nu'$
\State \textbf{Parameters:} Weights $\alpha \in [0, 1], \beta \ge 0$, entropic regularisation coefficients $\varepsilon_v, \varepsilon_e > 0$.
\State Initialise couplings:  $\pi^v_1 \gets \mu \otimes \mu', \: \pi^e_1 \gets \tilde{\nu} \otimes \tilde{\nu}'.$
\For{$t = 1, 2, \dots, \texttt{max\_iter}$}
\State $M^v \gets \nabla_v \Lc(\pi_t^v, \pi_t^e)$ using \eqref{eq:grad_L_v}
\State $\pi^v_{t+1} \gets \min_{\pi \in \Pi(\mu, \mu')} \: \langle M^v, \pi \rangle + \varepsilon_v \KL(\pi | \mu \otimes \mu')$  
\State $M^e \gets \nabla_e \Lc(\pi_t^v, \pi_t^e)$ using \eqref{eq:grad_L_e}
\State $\pi^e_{t+1} \gets \min_{\xi \in \Pi(\tilde{\nu}, \tilde{\nu}')} \: \langle M^e, \xi \rangle + \varepsilon_v \KL(\xi | \tilde{\nu} \otimes \tilde{\nu}')$
\EndFor
\State \textbf{Output:} couplings $(\pi^v, \pi^e)$
\end{algorithmic}
\label{alg:tpot_entropic}
\end{algorithm}

\begin{algorithm}
\caption{Unregularised TpOT: Block-coordinate descent}
\begin{algorithmic}[1]
\State \textbf{Input:} Pairwise affinity matrices $C, C'$, persistence diagrams $D, D'$, incidence matrices $\omega, \omega'$, probability distributions $\mu, \mu'$, positive measures $\nu, \nu'$
\State \textbf{Parameters:} Weights $\alpha \in [0, 1], \beta \ge 0$.
\State Initialise couplings:  $\pi^v_1 \gets \mu \otimes \mu', \: \pi^e_1 \gets \tilde{\nu} \otimes \tilde{\nu}'.$
\For{$t = 1, 2, \dots, \texttt{max\_iter}$}
\State $\pi^v_{t+1} \gets \min_{\pi \in \Pi(\mu, \mu')} \alpha \langle L(C, C'), \pi \otimes \pi \rangle + \beta  \langle  L(\omega_H, \omega_{H'}) \otimes \pi^e_t, \pi \rangle$
\State $M^e \gets \pi^v_{t+1} \otimes L(\omega_H, \omega_{H'}) + (1-\alpha) \tilde{C}(D, D')$
\State $\pi^e_{t+1} \gets \min_{\xi \in \Pi(\tilde{\nu}, \tilde{\nu}')} \: \langle M^e , \xi \rangle$ 
\EndFor
\State \textbf{Output:} couplings $(\pi^v, \pi^e)$
\end{algorithmic}
\label{alg:tpot_unreg}
\end{algorithm}

Note that the standard entropic regularisation term $\text{KL}(\pi_e | \tilde{\nu} \otimes \tilde{\nu}')$ yields inhomogeneity (see for instance~\cite{lacombe2023homogeneous}), and thus, the entropic term may dominate when the cardinality of the persistent diagram increases. While we believe that the approach used in this preliminary work is sufficient for the current analysis, this effect is worth investigating further and could be an important focus for future research.

\section{Characterisation of geodesics}\label{sec:characterization_of_geodesics}

Geodesic properties have been extensively studied for the space of measure networks with the Gromov-Wasserstein distance~\cite{sturm2023space,chowdhury2020gromov,memoli2023characterization} and the space of persistence diagrams under Wasserstein distance~\cite{turner2014frechet,chowdhury2019geodesics}. In this section, we derive similar results for the $\TpOT$ metric of Definition \ref{def:TpOT} when $p=2$. These results are of theoretical interest, but we plan to explore practical implications in future work; for example, the geodesic structure and curvature properties can be used to study Fr\'{e}chet means of ensembles of topological networks, extending ideas in~\cite{turner2014frechet,chowdhury2020gromov}. We begin by recalling some definitions from metric geometry---see~\cite{burago2022course} as a general reference.

\subsection{Metric Geometry Concepts}
Consider a metric space $(X,d)$. A \emph{geodesic} between points $x,y \in X$ is defined as a path $\gamma: [0,1] \longrightarrow X$ with $\gamma(0) = x$, $\gamma(1) = y$ and such that, for every $0 \leq s \leq t \leq 1$, we have $d(\gamma(s), \gamma(t)) = (t-s)d(x,y)$. In fact, it suffices to show that $d(\gamma(s), \gamma(t)) \leq (t-s)d(x,y)$ always holds, as the reverse inequality then follows for free---see, e.g.,~\cite[Lemma 1.3]{chowdhury2018explicit}.

We say that $(X,d)$ is a \textit{geodesic} metric space if, for any pair of points $x,y \in X$, there exists a geodesic $\gamma$ with $\gamma(0) = x$ and $\gamma(1) = y$. Further, we say that $(X,d)$ is \textit{uniquely geodesic} if for any pair of points $x,y \in X$ the geodesic connecting them is unique.

For $(X,d)$ a geodesic space, we say that it has \textit{curvature bounded below by zero}~\cite[Section 4.2]{sturm2023space} if, for every geodesic $\gamma: [0,1] \longrightarrow X$ and every point $x \in X$, the following holds:
\[
  d(\gamma(t),x)^2 \geq (1-t) d(\gamma(0),x)^2 + t d(\gamma(1),x)^2 - t(1-t)d(\gamma(0),\gamma(1))^2, \quad 0 \leq t \leq 1.
\]
Intuitively, this says that geodesic triangles in $X$ are always ``thicker" than the corresponding triangles in (flat) Euclidean space.

\subsection{Geodesics between measure persistence diagrams }

The space of persistence diagrams with the $p$-Wasserstein distance is well studied, and it is known that it admits geodesics which are essentially linear interpolations of diagrams~\cite{turner2014frechet,chowdhury2019geodesics}. What follows is a discussion of existence and uniqueness of geodesics in the case of measure persistence diagrams; these results are a straightforward generalisation of known results for persistence diagrams.

Consider two measure persistence diagrams $\nu_0$ and $\nu_1$, and $\xi \in \Pi_{\text{adm}}(\nu_0, \nu_1)$ a coupling between them. For $t \in [0,1]$ consider the map $\phi_t:  \overline{\Lambda} \times \overline{\Lambda} \to \overline{\Lambda}, \: \phi_t(\lambda, \lambda') = (1-t)\lambda + t\lambda'$. Finally, we define 
\begin{equation}\label{eqn:induced_measure}
    \nu_t^\xi := (\phi_t)_{\#} \xi.
\end{equation}

\begin{proposition}\label{prop:geopd}
  Consider measure persistence diagrams $\nu_0, \nu_1$, and $\xi$ an optimal coupling that realises $d^\mathrm{MPD}_{\mathrm{W},p}(\nu, \nu')$. The path  $\gamma(t) =  \nu_t^\xi$  defines a geodesic between  $\nu_0, \nu_1$. 
\end{proposition}

\begin{proof}
    Set $\xi_{st}$ as the coupling in $\Pi_{\rm adm}(\nu_s^\xi, \nu_t^\xi)$ given by $\xi_{st} = (\phi_s \times \phi_t)_\# \xi$, where we use the map
    \begin{align*}
        \phi_s \times \phi_t: (\overline{\Lambda} \times \overline{\Lambda}) &\rightarrow  (\overline{\Lambda} \times \overline{\Lambda}) \\
        (\lambda, \lambda') &\mapsto (\phi_s(\lambda, \lambda'),\phi_t(\lambda, \lambda')).
    \end{align*}
    By sub-optimality, we have that 
    \[ 
    d^\mathrm{MPD}_{\mathrm{W},p}(\nu_s^\xi, \nu_t^\xi)^p \leq \int_{\overline{\Lambda}^2}  \| \phi_s(\overline{\lambda})-\phi_t(\overline{\lambda}) \|^p_p \:\diff\xi_{st}(\overline{\lambda})  = \int_{\overline{\Lambda}^2}  \| \phi_s-\phi_t \|^p_p \: \diff \xi_{st},
    \]
    where $\overline{\lambda} = (\lambda, \lambda')$. From the equality $$ \|  \phi_s(\overline{\lambda})-\phi_t(\overline{\lambda}) \|^p_p = \| (t-s)\lambda - (t-s)\lambda' \|^p_p $$ we have 

\[
  \int_{\overline{\Lambda}^2}  \| \phi_s-\phi_t \|^p_p \;\; \diff\xi_{st}  = (t-s)^p \int_{\overline{\Lambda}^2}  \| d_p  \|^p_p \: \diff\xi  = (t-s)^p d^\mathrm{MPD}_{\mathrm{W},p}(\nu_0, \nu_1)^p.
\]
Putting everything together, we have  $d^\mathrm{MPD}_{\mathrm{W},p}(\nu_s^\xi, \nu_t^\xi) \leq (t-s) d^\mathrm{MPD}_{\mathrm{W},p}(\nu_0, \nu_1)$, and this completes the proof.
\end{proof}

\begin{remark}\label{rmk:nonuniquegeoPD}
  Consider two measure persistence diagrams $\nu_0, \nu_1$. Take any path between them of the form $\nu_t = (\phi_t)_{\#}\xi$, where $\xi$ is an optimal coupling between $\nu_0, \nu_1$, and for every $(\lambda, \lambda') \in \overline{\Lambda} \times \overline{\Lambda}$, the function $t \mapsto \phi_t(\lambda,\lambda')$ traces out a geodesic between $\lambda$ and $\lambda'$ in $(\overline{\Lambda}, \|\cdot\|_p)$. Then, the same argument as in Proposition~\ref{prop:geopd} shows that $\nu_t$ is a geodesic between $\nu_0, \nu_1$. In particular, this shows that for $p=1$, $d^\mathrm{MPD}_{\mathrm{W},1}$ admits geodesics which are not of the form $\phi_t$.
\end{remark}

\begin{remark}
    This formula for geodesics in the space of measure persistence diagrams takes the same form as the well-known formula for geodesics in classical Wasserstein space---see, e.g.,~\cite[Theorem 7.2.2]{ambrosio2005gradient}. 
\end{remark}

\subsection{Geodesics in the space of measure topological networks}

We now focus on geodesics in $\faktor{\mathcal{P}}{\sim_w}$, using techniques which follow the ones employed by Sturm in the Gromov-Wasserstein setting~\cite{sturm2023space}. Let $P = ((X, k, \mu), (Y, \iota, \nu), \omega)$ and $P' = ((X', k', \mu'), (Y', \iota', \nu'), \omega')$ be topological networks and let $(\pi^v, \pi^e)$ be a pair of optimal couplings that realise the infimum in Definition \ref{def:TpOT}. We construct the topological network 
\begin{equation}\label{eq:TOPgeo}
P_t = \left( (\Tilde{X}, k_t, \pi^v), (\Tilde{Y}, \iota_t, \pi^e),  \omega_t\right),
\end{equation}
where $\Tilde{X} = X \times X'$ and  $\Tilde{Y} =\left( (Y \times Y') \cup (Y \times \partial_{Y'}) \cup (\partial_{Y} \times Y')  \right)$ are as in Proposition~\ref{prop:weakiso} (and we once again simplify notation by using expressions such as $Y \times \partial_{Y'}$ rather than $Y \times \{\partial_{Y'}\}$), and
\begin{align*}
  \omega_t ((x,x'), (y,y')) &= (1-t)\omega(x,y) + t\omega'(x',y'), \\
     k_t ((x_1,x_1'), (x_2, x_2')) &= (1-t)k(x_1,x_2) + tk'(x'_1,x'_2), \\
     \iota_t(y,y') &= (1-t)\iota(y) + t\iota'(y').
\end{align*}
Note that $(\iota_t)_\# \pi^e = \nu_t^{(\iota \times \iota')_\#\pi^e}$ as in (\ref{eqn:induced_measure}).
In the following, we use $[P]$ to denote the equivalence class of $P \in \mathcal{P}$ with respect to $\sim_w$. 

\begin{theorem}\label{thm:geosame}
  For $P,P' \in \mathcal{P}$, the path $\gamma: [0,1] \to \faktor{\mathcal{P}}{\sim_w}$ defined by $\gamma_t = [P_t]$ defines a geodesic between $[P]$ and  $[P']$, with respect to $d_{\TpOT,p}$.
\end{theorem}

\begin{proof}
Throughout the proof, let $\pi^v \in  \Pi(\mu, \mu')$ and $\pi^e \in \Pi_{\rm adm}(\nu, \nu')$ be optimal couplings for $P, P'$. We fix $0 \leq s \leq t \leq1$, and we use the notation $\Tilde{y} = (y,y') \in \Tilde{Y}$, $\Tilde{x} = (x,x') \in \Tilde{X}$, and $\partial = (\partial_Y,\partial_{Y'})$. 
Consider now any pair of couplings $\pi, \xi$ for $P_s, P_t$. By sub-optimality, we have 
\begin{align}
  \begin{split}
    d_{\TpOT, p}(P_s, P_t)^p &\leq \int_{\Tilde{X}^2 \times \Tilde{Y}^2 } | \omega_s( \Tilde{x}_1, \Tilde{y}_1 ) -\omega_t(\Tilde{x}_2, \Tilde{y}_2) |^p \: \diff \pi(\Tilde{x}_1, \Tilde{x}_2) \: \diff \xi(\Tilde{y}_1,\Tilde{y}_2)) \\
                             &\qquad \qquad +  \int_{\Tilde{X}^4} | k_s(\Tilde{x}_1, \Tilde{x}_2) -k_t(\Tilde{x}_3, \Tilde{x_4}) |^p \: \diff\pi(\Tilde{x}_1, \Tilde{x}_3) \: \diff\pi(\Tilde{x}_2, \Tilde{x}_4) \\
                             &\qquad \qquad  +  \int_{ (\Tilde{Y} \cup \partial) ^2}  \| \iota_s(\Tilde{y}_1)-\iota_t(\Tilde{y}_2) \|^p_p \: \diff \xi(\Tilde{y}_1,\Tilde{y}_2).
  \end{split} \label{eqn:TpOT_geodesic}
\end{align}

Let $\ones_{\pi^v}$ denote the identity coupling of $\pi^v$ to itself, that is, 
$\ones_{\pi^v} = (\mathrm{diag}_{\Tilde{X}})_\#(\pi^v)$, where $\mathrm{diag}_{\Tilde{X}}$ is the map 
\begin{equation*}
  \mathrm{diag}_{\Tilde{X}}: \Tilde{X}\rightarrow \Tilde{X} \times \Tilde{X}, \quad  \Tilde{x} \mapsto (\Tilde{x},\Tilde{x}).
\end{equation*}
Similarly, set $\ones_{\pi^e} = (\mathrm{diag}_{\Tilde{Y}})_\#(\pi^e)$, where $\mathrm{diag}_{\Tilde{Y}}$ is defined analogously. Note that we have that $\xi_{st} \coloneqq  (\iota_s \times \iota_t)_\# \ones_{\pi^e} =  (\phi_s \times \phi_t)_\# \pi^e $ as in the proof of Proposition~\ref{prop:geopd}.

By the same argument as in the proof of Proposition~\ref{prop:geopd}, we have that 
\begin{align*}
    \int_{(\Tilde{Y} \cup \partial)^2}  \| \iota_s(\Tilde{y})-\iota_t(\Tilde{y}) \|^p_p \: \diff \ones_{\pi^e}(\Tilde{y}, \Tilde{y}) &= (t-s)^p \int_{\overline{Y} \times \overline{Y'}}  \| \iota(y) - \iota'(y') \|^p_p \: \diff \pi^e(y , y'). 
\end{align*}
Similarly, from the definition of $k_s$ and $k_t$, we have 
\[
  | k_s(\Tilde{x}_1, \Tilde{x}_2) - k_t(\Tilde{x}_1, \Tilde{x}_2) |^p = | (t-s) k(x_1, x_2) - (t-s) k'(x'_1, x'_2) |^p,
\]
and it follows that
\begin{align*}
  &\int_{\Tilde{X} ^4} | k_s(\Tilde{x}_1, \Tilde{x}_2) - k_t(\Tilde{x}_1, \Tilde{x}_2) |^p \: \diff \ones_{\pi^v}(\Tilde{x}_1, \Tilde{x}_1) \: \diff \ones_{\pi^v}(\Tilde{x}_2,\Tilde{x}_2) \\
  &\hspace{1in} = (t-s)^p \int_{\Tilde{X} ^4} | k(x_1, x_2) - k'(x'_1, x'_2)  |^p \: \diff\ones_{\pi^v}(\Tilde{x}_1, \Tilde{x}_1) \: \diff\ones_{\pi^v}(\Tilde{x}_2, \Tilde{x}_2)  \\
  &\hspace{1in} = (t-s)^p \int_{\Tilde{X} ^2} | k(x_1,x_2) - k'(x_1',x_2') |^p \: \diff\pi^v(x_1, x_1')\:\diff\pi^v(x_2, x_2'),
\end{align*}
where we have applied the definition of $\ones_{\pi^v}$ in the last line. Applying the same arguments to the remaining term in \eqref{eqn:TpOT_geodesic},  we have
\[
  | \omega_s( \Tilde{x}, \Tilde{y} ) - \omega_t( \Tilde{x}, \Tilde{y})|^p = | (t-s) \omega(x, y) - (t-s)\omega'(x', y')) |^p
\]
and
\begin{align*}
    &\int_{\Tilde{X} ^2 \times \Tilde{Y}^2 } | \omega_s( \Tilde{x}, \Tilde{y} ) - \omega_t(\Tilde{x}, \Tilde{y}) |^p \: \diff \ones_{\pi^v}(\Tilde{x}, \Tilde{x}) \: \diff\ones_{\pi^e}(\Tilde{y}, \Tilde{y}) \\
    &\hspace{1in} = (t-s)^p \int_{\Tilde{X} ^2 \times \Tilde{Y}^2 } | \omega( x, y ) - \omega'(x', y') |^p \: \diff\ones_{\pi^v}(\Tilde{x}, \overline{x}) \: \diff\ones_{\pi^e}(\Tilde{y}, \Tilde{y}) \\
    &\hspace{1in} = (t-s)^p \int_{\Tilde{X} \times (\overline{Y} \times \overline{Y'}) } | \omega( x, y ) - \omega'(x', y') |^p \: \diff\pi^v(x, x') \: \diff\pi^e(y, y').
\end{align*}
Putting these together, we have that 
\begin{align*}
  d_{\TpOT, p}(P_s, P_t)^p &\leq (t-s)^p \left(  \int_{\Tilde{X} \times (\overline{Y} \times \overline{Y'}) } | \omega( x, y ) - \omega'(x', y') |^p \: \diff\pi^v(x, x') \: \diff\pi^e(y, y') \right. \\
                           & \hspace{.5in} + \int_{\Tilde{X} ^2} | k(x_1,x_2) - k'(x_1',x_2') |^p \: \diff\pi^v(x_1, x_1') \: \diff\pi^v(x_2, x_2') \\
                           & \left. \hspace{.5in} + \int_{\overline{Y} \times \overline{Y'}}  \| \iota(y) - \iota'(y') \|^p_p \: \diff \pi^e(y , y') \right) \\
                           &   = \left( (t-s) d_{\TpOT, p}(P_s, P_t) \right)^p,
                           \end{align*}
and the result follows.

\end{proof}

A geodesic of the form described in Equation~\ref{eq:TOPgeo} will be called a \textit{convex geodesic}. A natural question is whether all geodesics in $\faktor{\mathcal{P}}{\sim_w}$ are convex. We will show in Theorem \ref{thm:geounique} that this is the case for $p=2$. As a first observation, the proof of Theorem \ref{thm:geosame} easily generalises to give the following corollary. This corollary, in particular, shows that if $p=1$ then there are geodesics which are not convex (c.f.~Remark \ref{rmk:nonuniquegeoPD}). 

\begin{corollary}\label{geo:unique}
Consider the path $Q^t$ between topological networks $Q^0$ and $Q^1 $ given by  $$Q^t =  \left( (\Tilde{X}, k_t, \pi^v), (\Tilde{Y}, \iota_t, \pi^e),  \omega_{t}  \right),$$ where $k_t, \omega_{t}, \pi^v,\pi^e , \Tilde{X},\Tilde{Y}$ are as in Proposition~\ref{thm:geosame}, and $\iota_t$ is such that the map $$ t \mapsto \iota_t((y^0,y^1)) $$ defines a geodesic between $\iota^0(y)$ and $\iota^1(y^1)$ in $(\overline{\Lambda}, \| \cdot \|_p)$ as $t$ varies in $[0,1]$ (as in Remark~\ref{rmk:nonuniquegeoPD}) for $\pi^e$-almost every $\Tilde{y} =(y^0,y^1) \in \Tilde{Y} $. Then $Q^t$ defines a geodesic $[Q^t]$ in $\faktor{\mathcal{P}}{\sim_w}$.

\end{corollary}

While some of the results below extend to $p > 1$, from now on, we focus on the case where $p = 2$, and write $d_{\TpOT}$ in place of $d_{\TpOT, p}$. This is for the sake of simplifying notation, and because we will use $p=2$ in computational implementations and examples below.   

\begin{theorem}\label{thm:geounique}
  All geodesics in $\left(\faktor{\mathcal{P}}{\sim_w}, d_{\TpOT}\right)$ are convex.
\end{theorem}
\begin{proof}
  The proof follows by adapting techniques from~\cite[Theorem 3.1]{sturm2023space} and~\cite[Theorem~10]{chowdhury2019geodesics}. 
  Let $[Q^t]$ be an arbitrary geodesic, with  $$  Q^t = \left( (X^t, k^t, \mu^t), (Y^t, \iota^t, \nu^t), \omega^t \right)$$
  Pick a dyadic decomposition of the unit interval $t_0 = 0, t_1 = \frac{1}{2^n}, \cdots, t_i = \frac{i}{2^n}, \cdots, t_{2^n} = 1$. For each $i \in \{0, \cdots , 2^n \}$, let $\pi^v_i, \pi^e_i $ be optimal couplings for $Q^{t_i}, Q^{t_{i-1}}$. Consider the gluings 
  \[
    \overline{\pi^v} = \pi^v_0 \boxtimes \pi^v_{\frac{1}{2^n}} \boxtimes \cdots \boxtimes \pi^v_1 \quad \mbox{and} \quad \overline{\pi^e} = \pi^e_0 \boxtimes \pi^e_{\frac{1}{2^n}} \boxtimes \cdots \boxtimes \pi^e_1
  \]
  (see Lemma \ref{lem:gluing} and the ensuing discussion) and the couplings 
  \[
    \pi^v = (p_0 \times p_1)_{\#}\overline{\pi^v} \quad \mbox{and} \quad \pi^e = (p_0 \times p_1)_{\#}\overline{\pi^e},
  \]
  with $p_i: X^0 \times X^{t_1} \times \cdots \times X^1 \rightarrow X^i$ projection on the $i$th factor. Then, by sub-optimality, we have
  \begin{align}
    d_{\rm TpOT}(Q^0, Q^1)^2 & \leq  \| \omega^0 - \omega^1 \|^2_{L^2(\pi^v \otimes \pi^e)} +  \| k^0 - k^1 \|^2_{L^2(\pi^v \otimes \pi^v)} +\| d_2 \circ (\iota^0, \iota^1) \|^2_{L^2(\pi^e )} \nonumber \\
                             &=: A +  B +  C \label{eq:uniquesplit}
  \end{align}
  For any choice of $t \in \{0, 1/2^n, 2/2^n, \cdots , 1 \}$, let 
  \[
    \xi^v_t = (p_0 \times p_1 \times p_t)_{\#}\overline{\pi^v} \quad \mbox{and} \quad \xi^e_t = (p_0 \times p_1 \times p_t)_{\#}\overline{\pi^e}.
  \]
  We now estimate \eqref{eq:uniquesplit} term-by-term. First observe that
  \begin{align}
    A & = \| \omega^0 - \omega^1 \|^2_{L^2(\pi^v \otimes \pi^e)} = \left\| t \left( \frac{1}{t}(\omega^0 - \omega^t) \right) + (1-t) \left( \frac{1}{1-t}(\omega^t - \omega^1) \right) \right\|^2_{L^2(\xi^v_t \otimes \xi^e_t)} \nonumber \\
      &= \frac{1}{t} \| \omega^0 - \omega^t \|^2_{L^2(\xi^v_t \otimes \xi^e_t)} + \frac{1}{1-t}\| \omega^t - \omega^1 \|^2_{L^2(\xi^v_t \otimes \xi^e_t)} \nonumber \\
      &\hspace{1.5in} - \frac{1}{t(1-t)} \| (1-t)(\omega^0 - \omega^t) -t(\omega^t - \omega^1) \|^2_{L^2(\xi^v_t \otimes \xi^e_t)}, \label{eq:uniqueomega}
  \end{align}
  with the last line following by the general identity
  \[
    | ta + (1-t)b |^2 = t |a|^2 + (1-t)|b|^2 - t(1-t)|a-b|^2,
  \]
  applied to $a = \frac{1}{t}(\omega^0 - \omega^t)$ and $b = \frac{1}{1-t}(\omega^t - \omega^1)$, pointwise. Similarly,  
  \begin{align}
    \begin{split}
        B & = \frac{1}{t}\| k^0 - k^t \|^2_{L^2(\xi^v_t \otimes \xi^v_t)} + \frac{1}{1-t}\| k^t - k^1 \|^2_{L^2(\xi^v_t \otimes \xi^v_t)} \\
        &\hspace{1.5in} - \frac{1}{t(1-t)}\| (1-t)(k^0 - k^t) -t(k^t - k^1) \|^2_{L^2(\xi^v_t \otimes \xi^v_t)}
    \end{split}
    \label{eq:uniqued}
  \end{align}
  and 
  \begin{align}
      \begin{split}
        C &=  \frac{1}{t}\| \iota^0 - \iota^t \|^2_{L^2( \xi^e_t)} + \frac{1}{1-t}\| \iota^t - \iota^1 \|^2_{L^2(\xi^e_t)} \\
        &\hspace{1.5in} - \frac{1}{t(1-t)}\| (1-t)(\iota^0 - \iota^t) -t(\iota^t - \iota^1) \|^2_{L^2(\xi^e_t)}.
    \end{split}
    \label{eq:uniquePD}
  \end{align}
  Recalling that $t = k2^{-n}$ for some $k$, the first term in \eqref{eq:uniqueomega} satisfies
  \begin{align}
    \frac{1}{t}\left\|\omega^0 - \omega^t \right\|_{L^2(\xi_t^v \otimes \xi_t^e)}^2 &= 2^n \cdot \frac{1}{k} \left\|\omega^0 - \omega^{k2^{-n}} \right\|_{L^2(\xi_t^v \otimes \xi_t^e)}^2 \\l
    &=  2^n \cdot \frac{1}{k} \left\|\sum_{\ell = 1}^k (\omega^{(\ell-1)2^{-n}} - \omega^{\ell 2^{-n}}) \right\|_{L^2(\xi_t^v \otimes \xi_t^e)}^2 \nonumber \\
    &\leq 2^n \cdot \frac{1}{k} \left(\sum_{\ell = 1}^k \left\|\omega^{(\ell-1)2^{-n}} - \omega^{\ell 2^{-n}}\right\|_{L^2(\xi_t^v \otimes \xi_t^e)}\right)^2  \label{eqn:geodesic_uniqueness_3}\\
    &\leq 2^n \sum_{\ell = 1}^k \left\|\omega^{(\ell-1)2^{-n}} - \omega^{\ell 2^{-n}} \right\|_{L^2(\xi_t^v \otimes \xi_t^e)}^2, \label{eqn:geodesic_uniqueness_4} 
  \end{align}
  where \eqref{eqn:geodesic_uniqueness_3} follows by the triangle inequality for the $L^2$-norm and \eqref{eqn:geodesic_uniqueness_4} is Jensen's inequality. Applying the same argument to the second term of \eqref{eq:uniqueomega}, as well as to \eqref{eq:uniqued} and \eqref{eq:uniquePD}, it follows that 
  \begin{align}
    \begin{split}\label{eq:unique1}
    A & \leq 2^n \sum_{j =1}^{2^n} \| \omega^{(j-1)2^{-n}} - \omega^{j2^{-n}} \|_{L^2(\pi^v_j \otimes \pi^e_j)}  \\
    &\hspace{1.5in} -  \frac{1}{t(1-t)}\| (1-t)(\omega^0 - \omega^t) -t(\omega^t - \omega^1) \|^2_{L^2(\xi^v_t \otimes \xi^e_t)}, 
    \end{split} \\
    \begin{split}\label{eq:unique2} 
    B& \leq 2^n \sum_{j =1}^{2^n} \| k^{(j-1)2^{-n}} - k^{j2^{-n}} \|_{L^2(\pi^v_j \otimes \pi^v_j)} \\
    &\hspace{1.5in} - \frac{1}{t(1-t)}\| (1-t)(k^0 - k^t) -t(k^t - k^1) \|^2_{L^2(\xi^v_t \otimes \xi^v_t)}, 
    \end{split} \\
    \begin{split}\label{eq:unique3}
    C& \leq 2^n \sum_{j =1}^{2^n} \| \iota^{(j-1)2^{-n}} - \iota^{j2^{-n}} \|_{L^2(\pi^e_j )} \\
    &\hspace{1.5in} - \frac{1}{t(1-t)}\| (1-t)(\iota^0 - \iota^t) -t(\iota^t - \iota^1) \|^2_{L^2(\xi^e_t)}.
    \end{split}
  \end{align}
  From \eqref{eq:unique1}, \eqref{eq:unique2}, \eqref{eq:unique3}, we deduce that
  \begin{align}
    &d_{\rm TpOT}(Q^0, Q^1)^2  \nonumber \\
    &\leq A + B + C \nonumber \\
                              \begin{split}
                                &\leq \sum_{j =1}^{2^n} \left( \| \omega^{(j-1)2^{-n}} - \omega^{j2^{-n}} \|_{L^2(\pi^v_j \otimes \pi^e_j)} + \| k^{(j-1)2^{-n}} - k^{j2^{-n}} \|_{L^2(\pi^v_j \otimes \pi^v_j)} \right.+ \\
                                &\hspace{3.5in}\| \left. \iota^{(j-1)2^{-n}} - \iota^{j2^{-n}} \|_{L^2(\pi^e_j )} \right)  \\ 
                              & \hspace{.25in} - \frac{1}{t(1-t)} \left( \| (1-t)(\omega^0 - \omega^t) -t(\omega^t - \omega^1) \|^2_{L^2(\xi^v_t \otimes \xi^e_t)} \right. \\
                              &\hspace{1.25in} \left.+ \| (1-t)(\omega^0 - \omega^t) -t(\omega^t - \omega^1) \|^2_{L^2(\xi^v_t \otimes \xi^e_t)} \right. \\
                              &\left. \hspace{2.25in} + \| (1-t)(\iota^0 - \iota^t) -t(\iota^t - \iota^1) \|^2_{L^2(\xi^e_t)}\right) 
                              \end{split} \nonumber \\ 
                              \begin{split}\label{eq:uniquefinal} 
                             &= d_{\rm TpOT}(Q^0, Q^1)^2  \\ 
                              & \hspace{.25in} - \frac{1}{t(1-t)} \left( \| (1-t)(\omega^0 - \omega^t) -t(\omega^t - \omega^1) \|^2_{L^2(\xi^v_t \otimes \xi^e_t)} \right. \\
                              &\hspace{1.25in} \left.+ \| (1-t)(\omega^0 - \omega^t) -t(\omega^t - \omega^1) \|^2_{L^2(\xi^v_t \otimes \xi^e_t)} \right. \\
                              &\left. \hspace{2.25in} + \| (1-t)(\iota^0 - \iota^t) -t(\iota^t - \iota^1) \|^2_{L^2(\xi^e_t)}\right),
                              \end{split}
  \end{align}
  where \eqref{eq:uniquefinal} follows by the assumption that $[Q^t]$ is a geodesic and by the optimality of $\pi^e_t, \pi^v_t$. This inequality implies that $\pi_v, \pi_e$ are optimal for $Q^0, Q^1$, and that 
  \begin{align*}
    &\| (1-t)(\omega^0 - \omega^t) -t(\omega^t - \omega^1) \|^2_{L^2(\xi^v_t \otimes \xi^e_t)} \\
    & \hspace{1in} + \| (1-t)(\omega^0 - \omega^t) -t(\omega^t - \omega^1) \|^2_{L^2(\xi^v_t \otimes \xi^e_t)} \\
    & \hspace{2in} + \| (1-t)(\iota^0 - \iota^t) -t(\iota^t - \iota^1) \|^2_{L^2(\xi^e_t)}   =0 
  \end{align*}
  for every dyadic number $t$. This implies that $d_{\rm TpOT}(Q^t, P_t) = 0$ for every dyadic number, with $P_t$ as in Theorem~\ref{thm:geosame}. Continuity of the function $t \mapsto Q^t$ and density of dyadic numbers in $[0,1]$ complete the proof. 
\end{proof}

From the characterisation of geodesics provided above, we easily deduce curvature bounds for $\faktor{\mathcal{P}}{\sim_w}$.

\begin{theorem}\label{thm:nonnegative_curvature}
  The metric space $\left( \faktor{\mathcal{P}}{\sim_w},d_\TpOT \right)$ has curvature bounded below by zero.
\end{theorem}

\begin{proof}
    Let $[P^t]$ be a geodesic in $\faktor{\mathcal{P}}{\sim_w}$ between $P^0$ and $P^1$. By Theorem~\ref{thm:geounique} we can assume that 
    \[
   P_t = \left( (\Tilde{X}, k_t, \pi^v), (\Tilde{Y}, \iota_t, \pi^e),  \omega_t\right),
    \]
    as in~\eqref{eq:TOPgeo}; in particular, $\pi^v, \pi^e$ are optimal for $P^0, P^1$. Let $ P' = \left( (X', k', \pi'), (Y', \iota', \nu'),  \omega' \right)$ be an arbitrary topological network, and let $\xi^v, \xi^e$ be optimal for $P^t$ and $P'$.  Then we have
    \begin{align}
        & d_{\rm TpOT, 2}(P^t, P')^2  + t(1-t)d_{\rm TpOT, 2}(P^0, P^1)^2   \nonumber \\
        &= \| \omega_t - \omega' \|^2_{L^2(\xi^v \otimes \xi^e)}  + \| k_t - k' \|^2_{L^2(\xi^v \otimes \xi^v)}  +  \| d_2 \circ (\iota_t, \iota')\|^2_{L^2(\xi^e )}   \nonumber \\
        &\quad + t(1-t) \left( \| \omega_0 - \omega_1 \|^2_{L^2(\pi^v \otimes \pi^e)} +  \| k_0 - k_1 \|^2_{L^2(\pi^v \otimes \pi^v)} + \| d_2 \circ (\iota_0, \iota_1)\|^2_{L^2(\pi^e )}  \right) \nonumber \\ 
        &= \| \omega_t - \omega' \|^2_{L^2(\xi^v \otimes \xi^e)}  +  \| k_t - k' \|^2_{L^2(\xi^v \otimes \xi^v)}  +   \| d_2 \circ (\iota_t, \iota')\|^2_{L^2(\xi^e )}   \nonumber \\
        &\quad + t(1-t) \left(  \| \omega_0 - \omega_1 \|^2_{L^2(\xi^v \otimes \xi^e)} +  \| k_0 - k_1 \|^2_{L^2(\xi^v \otimes \xi^v)} + \| d_2 \circ (\iota_0, \iota_1)\|^2_{L^2(\xi^e )}  \right) \label{eq:curvature} \\ 
        &= (1-t)\left(  \| \omega_0 - \omega' \|^2_{L^2(\xi^v \otimes \xi^e)} +  \| k_0 - k' \|^2_{L^2(\xi^v \otimes \xi^v)} +   \| d_2 \circ (\iota_0, \iota')\|^2_{L^2(\xi^e )}  \right) \nonumber \\
        &\quad + t \left( \| \omega_1 - \omega' \|^2_{L^2(\xi^v \otimes \xi^e)} +  \| k_1 - k' \|^2_{L^2(\xi^v \otimes \xi^v)} +   \| d_2 \circ (\iota_1, \iota')\|^2_{L^2(\xi^e )}  \right) \label{eq:curvature2} \\
        &\geq (1-t)d_{\rm TpOT, 2}(P^0, P')^2  + td_{\rm TpOT, 2}(P^1, P')^2, \label{eq:curvature3}
    \end{align}
where \eqref{eq:curvature} follows by marginal properties of $\xi^v$ and $\xi^e$, \eqref{eq:curvature2} by definition of $w_t, k_t$ and $\nu_t$, and \eqref{eq:curvature3} by sub-optimality. 
\end{proof}

\section{Numerical examples}\label{sec: results}

In this section, we demonstrate our computational framework and its capabilities on a range of numerical examples. In what follows, persistent homology computations are performed using the Julia package \texttt{Ripserer.jl}~\cite{vcufar2020ripserer} with the Vietoris-Rips filtration. We emphasise that the choice of filtration, representatives, and construction of the hypergraph incidence function $\omega$ influences the resulting transport plan and analysis. 

\begin{figure}[h!]
  \centering
  \includegraphics[width=\textwidth]{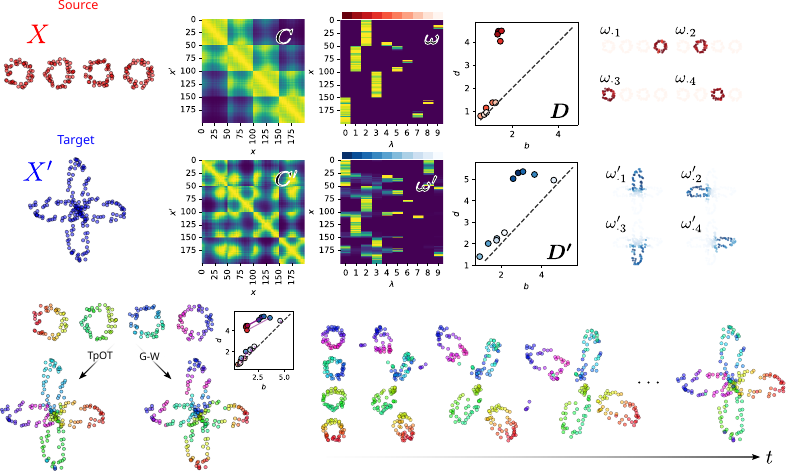}
  \caption{\textbf{(a)} Source point cloud $X$, \textbf{(b)} (i) Pairwise affinity matrix $C$ encoding geometric information; (ii) PH-hypergraph incidence matrix $\omega$ encoding the membership of points to persistence features; (iii) 1-dimensional persistence diagram $D$, points coloured by lifespan; (iv) columns of $\omega$ for the top 4 persistence features. Node-level features in $\omega$ are computed using TOPF \cite{grande2024node}, keeping the top 10 features. \textbf{(c)} Target point cloud $X'$. \textbf{(d)} Same as (b). \textbf{(e)} TpOT ($\alpha = 0.5, \beta = 1$) and GW matching. Inset: matching $\pi^e$ of the persistence diagrams $D, D'$ as found by TpOT. \textbf{(f)} A geodesic between the point clouds.}\label{fig:tpot_circles}
\end{figure}

\subsection{Matching point clouds}

In Figure \ref{fig:tpot_circles}, we consider a source point cloud $X$ consisting of four noisy circles of uniform size (Figure~\ref{fig:tpot_circles}(a)) and a target point cloud $X'$ resembling a ``flower'' with four noisy ellipses as ``petals'' (Figure~\ref{fig:tpot_circles}(c)). The point clouds $X$ and $X'$, together with the uniform distribution on points, are then measure spaces. We choose to represent geometric information by taking the affinity matrices (also referred to as gauge functions) to be Gaussian kernels $C, C'$, i.e. $C_{ij} = K(x_i, x_j)$ and $C'_{ij} = K(x_i', x_j')$ where
\[
  K(x, y) = \exp\left(- \frac{\| x - y \|_2^2}{h^2} \right).
\]
The bandwidth $h$ is chosen for each set of points $\{x_1, \ldots, x_N \}$ such that $h^2 N^{-2} \sum_{ij} \| x_i - x_j \|_2^2 = 1$. 

We compute $1$-dimensional persistent homology using the Vietoris-Rips filtration, and keep the top 10 features by persistence. We find that $\mathcal{K}_{\rm VR}(X)$ contains 4 significant, almost identical, homology classes (one per circle), and the same is true for $\mathcal{K}_{\rm VR}(X')$ (one class per ellipse). A choice of generating cycle for each class in the diagrams yields the PH-hypergraphs $H_X = H(X, \mathcal{K}_{\rm VR}(X), g_X )$ and $H_{X'} = H(X', \mathcal{K}_{\rm VR}(X'), g_{X'})$ and  the corresponding topological networks (see Section \ref{sec:discrete})
\begin{align*}
  P_X & =  \left( (X, C, \mu), (Y,\iota,\nu),  \omega   \right) \\ 
  P_{X'} &=  \left( (X', C', \mu'), (Y',\iota',\nu'),  \omega' \right).
\end{align*}

As we point out earlier, the choice of pointwise representation of persistence features can affect results greatly. To improve robustness to noise, we employ the TOPF approach of \cite{grande2024node} to compute the PH-hypergraph membership function for each persistence feature. In Figure~\ref{fig:tpot_circles}(b) and (d), we see, from left to right: the target and source point clouds, their pairwise affinity matrix, the PH-hypergraphs incidence matrices $\omega, \omega'$, their persistence diagrams, and a representation of their representative cycles as encoded in the corresponding PH-hypergraph.
Using the algorithms we developed in Section \ref{sec:discrete}, we are able to simultaneously find matchings between points and topological features in $P_X$ and $P_{X'}$ as the parameters $(\alpha, \beta)$ vary. 

We compute TpOT matchings of $P_X$ and $P_{X'}$ using Algorithm \ref{alg:tpot_entropic} with $\varepsilon_v = 3 \times 10^{-3}$, $\varepsilon_e = 10^{-2}$, $\alpha = 0.5$, $\beta = 1.0$. 
For comparison, we compute the (unregularised) Gromov-Wasserstein matching of $X$ an $X'$ using the PythonOT package \cite{flamary2021pot}. 

In Figure~\ref{fig:tpot_circles}(e), points are coloured by their corresponding source points under the TpOT and Gromov-Wasserstein matchings respectively. From this we observe that topological features (i.e. closed loops) fail to be preserved by the GW matching. Instead, each of the circles in the source point cloud is split up among multiple ellipses in the target point cloud. The TpOT matching, on the other hand, simultaneously preserves local adjacency and topological features.

\begin{figure}[ht]
  \centering
  \includegraphics[width=\textwidth]{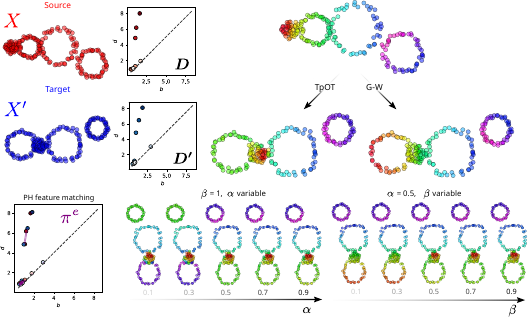}
  \caption{\textbf{(a)} Source point cloud $X$ and its persistence diagram $D$. \textbf{(b)} Target point cloud $X'$ and its persistent diagram $D'$. \textbf{(c)} TpOT ($\alpha = 0.5, \beta = 1$) and GW matching. \textbf{(d)} Matching $\pi^e$ of the persistence diagrams $D, D'$ found by TpOT. \textbf{(e)} The variation of TpOT matchings as the parameters $\alpha$ and $\beta$ vary.}\label{fig:noisy_disk}
\end{figure}

Given couplings $(\pi^v, \pi^e)$, Theorem~\ref{thm:geosame} allows us to explicitly construct a geodesic $[P_t], t \in [0, 1]$ between $P_X$ and $P_{X'}$. At each time $t \in [0, 1]$, the corresponding gauged measure space $P_t$ can be represented as a point cloud with a gauge function $\omega_t$ given by interpolation of $\omega$ and $\omega'$. In order to visualise the family of interpolating points, we consider the function $C_t((x, x'), (y, y')) \mapsto (1-t) C(x, y) + t C'(x', y')$ where $C(x, y)$ and $C'(x', y')$ are the squared Euclidean distances on the source and target point clouds respectively; this amounts to linear interpolation of $C$ and $C'$ in $L^2_{\pi^v \otimes \pi^v}((X \times X')^2)$.  
Since in practice we use an entropic regularisation, the resulting couplings $(\pi^v, \pi^e)$ are a priori strictly positive for all entries. We use the following heuristic sparsifying step, \rrrev{which was also used by \cite{han2023covariance}:} given a coupling $\pi_\varepsilon \in \Pi(\mu, \mu')$, we ``round'' it onto an extreme point of the coupling polytope by solving 
\[
    \max_{\pi \in \Pi(\mu, \mu')} \langle \pi_\varepsilon, \pi \rangle,
\]
which in itself amounts to solving a discrete optimal transport problem. This yields a sparse coupling approximating the solution found by entropic regularisation. 

At each value of $t$, positions of points in the interpolating point cloud are obtained by applying the multidimensional scaling (MDS) algorithm to $C_t$, followed by an alignment step to remove issues due to the invariance of MDS under rigid transformations. Figure~\ref{fig:tpot_circles}(f) shows a snapshot of this geodesic computed numerically for $\alpha = 0.5, \beta = 1$. We note that the geodesic almost perfectly recovers a local homeomorphism connecting each loop with their matched petal. 

\begin{figure}[ht]
  \centering
  \includegraphics[width=\textwidth]{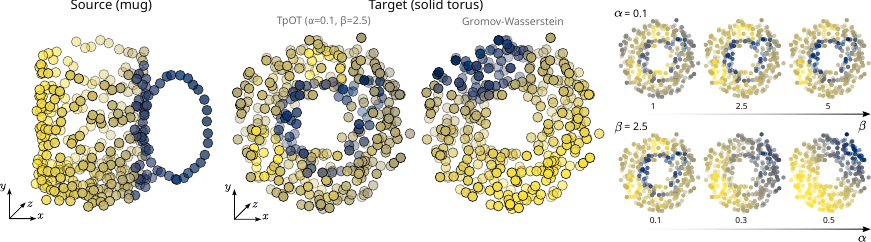}
  \caption{\textbf{(a)} Source point cloud: a noisy mug in $\mathbb{R}^3$. \textbf{(b)} Target point cloud, a noisy solid torus in $\mathbb{R}^3$, with matching induced by TpOT and GW distance, respectively. \textbf{(c)} The variation of TpOT matching as the parameters $\alpha$ and $\beta$ vary.}\label{fig:mug}
\end{figure}


In Figure \ref{fig:noisy_disk} we consider two point clouds $X, X'$ which have nearly indistinguishable persistence diagrams $D, D'$. As previously, we construct measure topological networks $P_X, P_{X'}$ with Gaussian affinity matrices and using TOPF with a Vietoris-Rips filtration \cite{grande2024node} for topological features. We compute the TpOT matchings as before (Algorithm \ref{alg:tpot_entropic}) with $\alpha = 0.5, \beta = 1, \varepsilon_v = 3 \times 10^{-3}, \varepsilon_e = 10^{-2}$.
The choice of parameters $\alpha$ and $\beta$ for balancing the weights for the TpOT terms (see Equation~\ref{eq:tpotdiscrete}) is a key factor influencing the behaviour of the matching. Of the three terms, the key one is the ``cross-term'' $\left\langle L(\omega, \omega'), \pi^v \otimes \pi^e \right \rangle$ corresponding to the co-cost function, since this determines the strength of coupling between the geometric and generator matchings. When this term is given a zero weight (\textit{i.e.}~when $\beta = 0$), the TpOT problem simplifies to that of independently finding geometric and generator matchings using a Gromov-Wasserstein and persistence diagram cost, respectively. Figure~\ref{fig:noisy_disk} shows the effect of varying $\beta$ on the example of Figure~\ref{fig:noisy_disk}. As $\beta$ increases, so does the strength of the coupling between these two terms. In the limit as $\beta \to \infty$, this term dominates the overall objective in \eqref{eq:tpotdiscrete}. In the limit we recover a partial matching variant of HyperCOT~\cite{chowdhury2023hypergraph}, in which hyperedges can be transported to a null edge (corresponding to the diagonal in the persistence diagram setting) for a cost equal to the squared $L^2$-norm of the corresponding incidence function. For positive $\beta$, the parameter $\alpha$ controls the relative contributions of the geometric distortion (measured in terms of a Gromov-Wasserstein loss, with coefficient $\alpha$) and topological distortion (measured in terms of transport cost on persistence features, with coefficient $1-\alpha$). For small values of $\alpha$ (e.g. $\alpha = 0.1$), Figure~\ref{fig:noisy_disk} the influence of the topological distortion in the PD-space is greater. As a result, TpOT matches points in cycles with similar lifespan. On the other hand, for higher values (e.g. $\alpha = 0.9$) the geometric distortion has greater influence, and we observe that the matching tends to preserve local proximity between loops, rather than their relative sizes. \\

\begin{figure}[ht]
  \centering
  \includegraphics[width=\textwidth]{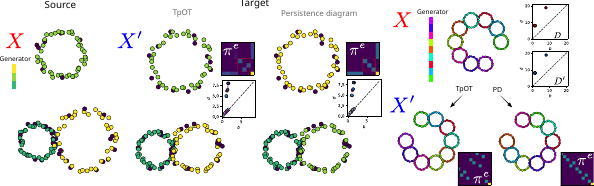}
  \caption{\textbf{(a)} A point cloud $X$ consisting of 3 noisy loops of various sizes. \textbf{(b)} On the left-hand side: a point cloud $X'$ consisting of the same 3 noisy loops as $X$, placed in different ways. Points forming the three persistent cycles in $\mathcal{K}_{\rm VR} (X)$ and $\mathcal{K}_{\rm VR} (X')$ are coloured according to TpOT-induced geometric matching. On the right-hand side: same as on the left, but points forming the three persistent cycles in $\mathcal{K}_{\rm VR} (X)$ and $\mathcal{K}_{\rm VR} (X')$ are coloured according to Wasserstein matching on persistent diagrams.  \textbf{(c)} A point cloud $X$ consisting of a chain of 9 noisy circles, matched to a copy of itself $X'$. Colour indicates cycle matching induced by TpOT-induced geometric matching and Wasserstein matching on persistent diagrams, respectively.}
  \label{fig:ex4}
\end{figure}

We show a more challenging example in Figure~\ref{fig:mug}, in which the objective is to transport a noisy point cloud $X$ resembling a mug (see Figure~\ref{fig:mug}(a)), into a noisy solid torus $X'$ (see Figure~\ref{fig:mug}(b)) whilst preserving the topological features found by persistent homology. In this context, matching of topological features amounts to mapping of the ``handle'' in $X$ into an \textit{essential} (i.e. not bounding a disk) closed curve in $X'$. The persistence diagrams of $X$ and $X'$ have only one significant class each. This implies that, outside of the points involved in the two chosen representative cycles, points in $X$ and $X'$ have almost trivial topological signal.
To construct measure topological networks $P_X, P_{X'}$, we construct the PH-hypergraph incidence functions $\omega, \omega'$ using the following smoothing procedure applied to a binary point-generator incidence function $\omega_{\mathrm{binary}}$:
  \[
    \omega_{\mathrm{smooth}} = \min_{x} \frac{1}{2} \| x - \omega_{\mathrm{binary}} \|_F^2 + \frac{\lambda}{2} \operatorname{tr}(x^\top L x) = (I + \lambda L)^{-1} \omega_{\mathrm{binary}},
  \]
  where we take $\lambda = 1$ and $L$ is a symmetrically normalised Laplacian operator constructed, in our case, from a Gaussian kernel on each point cloud. The pairwise affinity matrices $C, C'$ are also constructed using Gaussian kernels as done for the examples in Figures \ref{fig:tpot_circles}, \ref{fig:noisy_disk}. 

We compute the TpOT matchings using Algorithm \ref{alg:tpot_entropic} with $\varepsilon_v = 10^{-2}, \varepsilon_e = 10^{-2}$, and various choices of the parameters $\alpha, \beta$. For $\beta = 2.5$ and low values of $\alpha$ (e.g., $\alpha = 0.1$), the coupling $\pi^v$ matches points successfully maps the handle in an essential closed curve spanning the doughnut hole (Figure \ref{fig:mug}(b), left). When $\alpha$ increases (Figure~\ref{fig:mug}(c)), this effect is lost since the Gromov-Wasserstein component dominates.
On the other hand, fixing $\alpha = 0.1$ and allowing $\beta$ to vary, we find that the correspondence between generators is noisy for smaller values $\beta = 1$ but becomes stronger as $\beta$ increases. This can be understood as the increasing contribution of the coupling between points and topological features.

\begin{figure}[ht!]
\centering
\includegraphics[width=\textwidth]{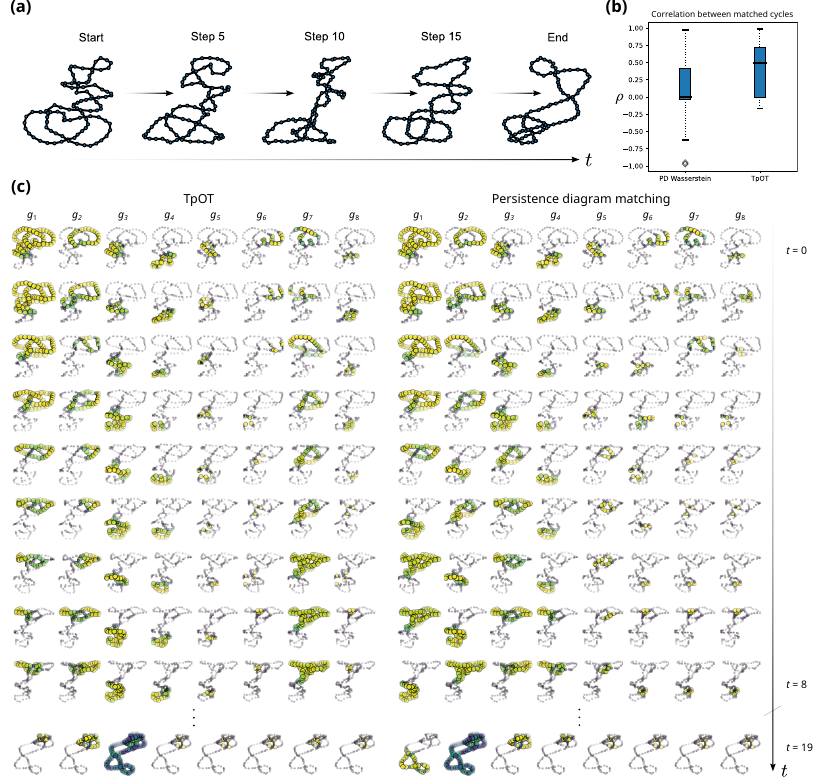}
\caption{\textbf{(a)} Snapshots of a trefoil undergoing thermal relaxation. \textbf{(b)} Box plot of correlation values between cycles and their images. \textbf{(c)} Step by step tracking of cycles via TpOT-induced geometric matching (left) and Wasserstein matching on persistence diagrams (right). }\label{fig:trefoil}
\end{figure}

\subsection{Geometric cycle matching}

The past few years have seen increasing efforts in addressing the problem of matching topological features across different systems, and many solutions have been proposed, including some inspired by techniques in optimal transport~\cite{li2023flexible, hiraoka2023topological, giusti2023signatures}. Perhaps, the simplest approach is to straightforwardly use the matching induced by Wasserstein-like distances on persistence diagrams, such as the bottleneck distance \cite{cohen2005stability} and the $d_{\mathrm{W},p}^\mathrm{PD}$ Wasserstein distances~\cite{chowdhury2019geodesics}. More nuanced solutions rely on statistical considerations~\cite{cohen2009persistent, reani2022cycle, garcia2022fast}, on the existence of maps between the initial data~\cite{bauer2013induced, gonzalez2020basis}, or, when such a map is unknown, on leveraging the algebraic topology of the PH construction to define a notion of dissimilarity between the underlying complexes~\cite{yoon2023persistent}. 

In our case, the topological feature coupling $\pi^e$ yields a matching between homology classes which is informed by the proximity of points creating the corresponding geometric cycles. We showcase this in Figure~\ref{fig:ex4} where we compare the standard Wasserstein matching on the persistence diagram space (``PD Wasserstein'') and our TpOT matchings. In Figure \ref{fig:ex4}(a-b) consider two point clouds $X$ and $X'$, each consisting of three cycles of different sizes and arranged differently in space. Consequently, the persistence diagrams of $X$ and $X'$ are nearly identical, despite the relative positioning of the cycles being different. Joint matching of geometric and topological features using TpOT produces a mapping between generators that preserves their layout. In contrast, Wasserstein matching on the persistence diagrams alone matches features by their persistence alone, breaking the spatial arrangement.
As a more extreme example, in Figure \ref{fig:ex4}(c) we consider two point clouds $X$, $X'$ with 9 identical loops each. As before, TpOT recovers a matching of generators that preserves their relative ordering while Wasserstein matching on the persistence diagrams produces a shuffled ordering. In all of these examples, measure topological networks $P_X, P_{X'}$ are constructed using affinity matrices of squared Euclidean distances and binary hypernetwork incidence function. TpOT matchings are computed as before using $\varepsilon_v = 3 \times 10^{-3}, \varepsilon_e = 10^{-2}, \alpha = 0.5, \beta = 1.0$.


Finally, we consider the problem of tracking topological features as a system evolves over time. Figure~\ref{fig:trefoil}(a) shows a piecewise linear spatial curve forming a trefoil knot as it undergoes thermal relaxation. The simulation consists of $20$ steps generated using KnotPlot~\cite{scharein2002interactive}. At each time $t_i$, the knot takes a spatial configuration $C_{i}$, that can be represented as the point cloud $X_{i}$ of its vertices. At each step, we construct from $X_i$ a measure topological network $P_{X_i}$ obtained by computing 1-dimensional PH. Similar to the example in Figure \ref{fig:mug}, we choose to use a Gaussian kernel for pairwise affinity matrices, and use the Laplacian-smoothed binary incidence matrices. We can then match homology classes between consecutive temporal instants $(t_i, t_{i+1})$ by either Wasserstein matchings on persistence diagrams or TpOT. Since the ground truth correspondence of points in the curve is known, in Figure \ref{fig:trefoil}(b) we show the correlation between generators. We find that the TpOT matching significantly improves the accuracy of generator matchings compared to PD Wasserstein.  Finally, in Figure \ref{fig:trefoil}(c) we show the images of significant generators under the mappings found by PD Wasserstein as well as TpOT. We find that the TpOT matching provides a clear improvement over the PD Wasserstein matching. 

Reliably tracking topological features as a system evolves over time offers the key advantage of identifying structural features resilient to dynamic changes. For spatial curves, this capability has clear implications for studying the folding and dynamics of (bio)polymers. In a polymer with a specific—potentially knotted—spatial configuration, are regions that take longer to untangle biophysically significant, and if so, why? A natural setting where this question is particularly relevant is the case of proteins. Folding intermediates are difficult to obtain, both experimentally and via simulations, and similar topological considerations can inform and enhance the analysis of protein folding.

\begin{remark}
Consider two point clouds, and the convex $GW$ geodesic connecting them (see~\cite{sturm2023space}). Consider now the Vietoris-Rips filtration along the geodesic path. A natural question is whether the path induced by the persistent diagrams follows a convex geodesic. A counterexample to this question is given by Figure~\ref{fig:tpot_circles}, as the GW-mapping between points suggests that the GW-geodesic ``breaks'' the loops, and thus, the four persistent homology classes. This can be verified by computing the path of Vietoris-Rips PDs obtained by inputting the various interpolation metrics. The same question applies for $TpOT$-geodesics. In this context, the answer depends on whether the $GW$-term contributes to cycle matching or not. When it doesn't (\textit{e.g.}~Figure~\ref{fig:tpot_circles}(E)), persistent diagrams computed from Vietoris-Rips filtration along the TpOT-geodesic do indeed follow the convex-geodesic in the persistent diagram space. When GW contributes (\textit{e.g.}~Figure~\ref{fig:noisy_disk} and Figure~\ref{fig:ex4}), then the persistent diagram path is non-geodesic. 
\end{remark}

\subsection{Higher-degree homology}
Note that the TpOT approach is not constrained to 1D PH; it can be applied more broadly to higher-dimensional PH data. The structure and complexity of PH-hypergraphs varies with the homology dimension~\cite{barbensi2022hypergraphs}, and for data lying in higher-dimensional spaces, the information from homology in higher degrees may provide more nuanced insights. As a concrete example, recent work by a subset of the authors and collaborators demonstrates that 2-dimensional cycles are effective in detecting damaging mutations in protein structures~\cite{madsen2023topological}. For clarity and conciseness, and because this setting is directly relevant to the experiments we present, we have focused on the 1D case.

\subsection{Computational complexity}
Our method computes an approximate solution of a non-convex quadratic program, which is NP-hard to solve exactly. Our iterative algorithm for finding an approximate solution relies on solving successive linear optimal transport subproblems, which are at least $O(n^3 \log n)$ in the exact case. The entropy regularised case is slightly more efficient, see~\emph{e.g.}~\cite{peyre2019computational}. We remark that our method currently stores full cost matrices in memory and is thus $O(n^2)$ in space; for large point clouds speedups could be used. Persistent homology computations are known to be computationally very expensive (see~\textit{e.g.}~\cite{distributing, reduced}). As an example, the computation of first-degree homology requires analysing $O(n^3)$ $2$-simplices. For these computations, we are using the \texttt{Julia} package \texttt{Ripserer}~\cite{vcufar2020ripserer}, one of the top performing softwares currently available.

\subsection{Generating cycles}
Constructing PH-hypergraphs, and thus the TpOT pipeline, depends on computing explicit cycles representing homology classes. It is well known that, even in the case of simplicial complexes, generating cycles are far from being unique, and that different choices can potentially bring important biases in the subsequent analysis~\cite{LiMinimalCycle}. Ideally, one would want to work with some \textit{minimal} generating basis, or have some sort of \textit{preferred} or \textit{best-practice} method to choose cycle representatives. Unfortunately, in general, finding minimal generators is an NP-hard problem~\cite{Dey2019, LiMinimalCycle}. In practice, we observe good results with the representatives found by the software \texttt{Ripserer}~\cite{vcufar2020ripserer}, a package which has enjoyed widespread adoption in the TDA community. In addition to that, the experimentally shown robustness of the hyperTDA method~\cite{barbensi2022hypergraphs} suggests that, when dealing with systems sufficiently complex from a topological point of view, the qualitative behaviour of a TpOT induced matching should be robust by changes to the representative cycles.  Further, the implementation of the approach by Grande and Schaub~\cite{grande2024node} assures an increased robustness of the overall pipeline. 

Given their potential and demonstrated usefulness, the search for new methods of producing ``good'' representatives, and the research into homology generators in general~\cite{ChenGenerators, Dey2019, emmett2015multiscale,  LiMinimalCycle, Obayashi2018VolumeOC, Ripser_involuted, barbensi2022hypergraphs}, is a very active sub-field of topological data analysis. New, important advancements in this area currently provide a wide variety of algorithms to compute representatives, that can all be implemented in the TpOT pipeline, depending on the specific goal or application.

\section{Conclusion}\label{sec:conclusions}

Inspired by recent advances in computational topology~\cite{barbensi2022hypergraphs} and optimal transport for structured objects~\cite{vayer2020fused, redko2020co,chowdhury2023hypergraph}, we propose to encode geometric and topological information about a point cloud jointly in a hypergraph structure. Our approach lays the foundation for a topology-driven analysis of the geometry of gauged measure spaces. As in the co-optimal setting~\cite{redko2020co,chowdhury2023hypergraph}, the TpOT problem outputs a pair of coupled transport plans. The first one provides a matching of the underlying metric spaces (point clouds) that preserves topological features. The second one matches the topological features (cycles representing persistent homology classes) in a way that maintains local metric structure. This geometric cycle matching represents a key novelty of our method, and, to the best of our knowledge, it is the first one taking into account the geometric nature of data, and the spatial interconnectivity of generating cycles.

\subsection*{Data availability}

Data and software used to produce results in this paper are available at the GitHub repository \url{https://github.com/zsteve/TPOT}

\subsection*{Acknowledgements}

SYZ gratefully acknowledge funding from the Australian Government Research Training Program and an Elizabeth and Vernon Puzey Scholarship. MPHS is funded through the University of Melbourne DRM initiative, through an ARC Laureate Fellowship (FL220100005) and acknowledges financial support from the Volkswagen Foundation through a ``Life?'' program grant. TN was partially supported by NSF grants DMS 2107808 and DMS 2324962.

\section*{Declarations}
The authors declare no conflict of interest or competing interests. 

\bibliographystyle{plain}

\end{document}